\documentclass[10pt]{amsart}

\usepackage{geometry,graphicx,amssymb,amsmath,amsbsy,eucal,amsfonts,mathrsfs,amscd,bm,tcolorbox,subfigure,algorithm}

\numberwithin{equation}{section}

\allowdisplaybreaks[3]

\newtheorem{theorem}{Theorem}[section]
\newtheorem{lemma}[theorem]{Lemma}

\newtheorem{corollary}[theorem]{Corollary}

\theoremstyle{definition}

\theoremstyle{remark}

\newcommand{\bg}{{\bld g}}
\newcommand{\bLam}{{\bld \Lambda}}
\newcommand{\blam}{{\bld \lambda}}

\newcommand{\bxi}{{\bld \xi}}
\newcommand{\eps}{\varepsilon}

\newcommand{\bl}{\bigl\langle}
\newcommand{\br}{\bigr\rangle}

\newcommand{\dive}{{\ensuremath\mathop{\mathrm{div}\,}}}

\newcommand{\pol}{\EuScript{P}}
\newcommand{\bpol}{\boldsymbol{\pol}}
\newcommand{\bld}[1]{\boldsymbol{#1}}

\newcommand{\bv}{\bld{v}}
\newcommand{\bw}{\bld{w}}

\newcommand{\bn}{\bld{n}}
\newcommand{\bu}{\bld{u}}

\newcommand{\bU}{\bld{U}}

\newcommand{\bV}{\bld{V}}

\newcommand{\bq}{\bld{q}}

\newcommand{\bH}{\bld{H}}
\newcommand{\bQ}{\bld{Q}}

\newcommand{\bmu}{{\bld{\mu}}}

\newcommand{\bphi}{\bm \phi}

\newcommand{\Om}{\Omega}

\newcommand{\bfeta}{{\bm \eta}}

\newcommand{\bpsi}{\bm \psi}
\newcommand{\Dt}{\Delta t}
\newcommand{\dep}{{\bm \eta}}

\newcommand{\la}{\langle}
\newcommand{\ra}{\rangle}

\newcommand{\pdt}{\partial_{\Delta t}}
\newcommand{\pt}{\partial_t}
\newcommand{\de}{\delta}

\title[Robin-Robin coupling]{Fully discrete
  loosely coupled
  Robin-Robin scheme for incompressible fluid-structure interaction:
  stability and error analysis}

\author[]{Erik Burman}
\address{Department of Mathematics, University College London, London, UK–WC1E 6BT, United Kingdom}
\email{e.burman@ucl.ac.uk}

\author[]{Rebecca Durst}
\address{Division of Applied Mathematics,
Brown University,
182 George Street,
Providence, RI 02912, USA}
\email{rebecca\_durst@brown.edu}

\author[]{Miguel A. Fern\'andez}
\address{Inria Paris, 75012 Paris, France -- Sorbonne Universite \& CNRS, UMR 7598 LJLL, 75005 Paris, France 
}
\email{miguel.fernandez@inria.fr}

\author[]{Johnny Guzm\'an}
\address{Division of Applied Mathematics,
Brown University,
182 George Street,
Providence, RI 02912, USA}
\email{johnny\_guzman@brown.edu}

\begin{document}

\maketitle

\begin{abstract}
We consider a fully discrete loosely coupled scheme for incompressible fluid-structure
interaction based on the time semi-discrete splitting method introduced in 
{\emph{[Burman, Durst \& Guzm\'an, 
  arXiv:1911.06760]}}.
The splittling method uses a Robin-Robin type coupling that allows for
a segregated solution of the solid and the fluid systems, without
inner iterations. For the discretisation in
space we consider piecewise affine continuous finite elements for all
the fields and ensure the inf-sup condition by using a
Brezzi-Pitk\"aranta type pressure stabilization. The interfacial fluid-stresses  
are evaluated in a variationally consistent fashion, that is shown to admit 
an equivalent Lagrange multiplier formulation.  We prove that the method is
unconditionally stable and robust with respect to the amount of added-mass in the system. 
Furthermore, we provide an error estimate that shows the
error in the natural energy norm for the system is $\mathcal O\big(\sqrt{T}(\sqrt{\Delta t} + h)\big)$ where $T$ is the final time, 
$\Delta t$ the time-step length and $h$ the space discretization  parameter.
\end{abstract}

\section{Introduction}
The computational solution of fluid-structure interaction problems remains
a challenging problem. Indeed the combination of the continuity
requirement of velocities and stresses across the interface with
the incompressibility constraint leads to a very stiff problem.
In order to be able to use optimised solvers and existing codes for the fluid and the
solid sub-systems and to simplify the handling of geometric nonlinearities
it is appealing to use a loosely coupled (or explicit coupling) where the solid and fluid systems are solved sequentially,
passing information across the coupling interface at discrete time
levels, without iterating between the sub-systems within one time-step.
This partitioned solution procedure has been very successful in the context of
aeroelasticity (see \cite{FLT98}), but in other applications, depending on
the geometry of the computational domain or the physical parameters, it has been shown to suffer from severe stability problems (see \cite{causin2005added}).
In particular, in
applications where the fluid-solid density ratio is close to one any
naive decoupling of the fluid-solid system to form a loosely coupled
scheme is known to be unstable.

In this paper, we revisit the loosely coupled scheme based on a
Robin-Robin type coupling for the coupling of an incompressible fluid
with a thick-walled solid, introduced in
\cite[Algorihm 4]{burman2014explicit}. Recently (see \cite{BDG19}), this
method was analysed in the time semi-discrete framework,
i.e. independent of any space mesh parameter, and was shown to be
stable independent of the fluid-solid density ratio. In particular, the
dependence of the Robin-coefficient on the inverse of the space mesh
parameter (and the pressure stabilizer necessary for the stability arguments of \cite{burman2014explicit}) were eliminated. The splitting
error of the scheme could then be shown to be $\mathcal O(\sqrt{T}\sqrt{\Delta t})$. Note the absence of exponential growth of
perturbations in time (with respect to $T$). The success in \cite{BDG19} relies on the fact that,
at the continuous level,  one can use the Robin condition in strong
form, provided the solution of the time-discretised fluid 
system is sufficiently regular. Indeed, if $\alpha \in \mathbb{R}^+$
denotes the Robin parameter, $\bu$, $\bq$, $\sigma_f(\bu,p)$ and $\sigma_s(\bfeta)$ are the fluid and solid velocities and stresses at
some time-levels, $*$ superscripts indicate a shift by one time-step, and $\Sigma$ stands for the
fluid-solid interface, the Robin type coupling
condition in the fluid formally reads
\[
\sigma_f(\bu,p)n_f+ \alpha \bu= \alpha  \bq +
\sigma_f(\bu^*,p^*)n_f\quad \mbox{on}\quad \Sigma.
\]
A key ingredient in the stability analysis of the time semi-discretized method
is to write
\begin{equation}\label{eq:comp_Robin}
\alpha (\bu - \bq) = \sigma_f(\bu^*,p^*)n_f - \sigma_f(\bu,p)n_f.
\end{equation}
Assuming sufficient regularity, the discrepancy in the velocities across the interface can then be
replaced by the increment of the stresses, which is used to obtain stability. It should  however be noted that, at the
discrete level, the relation \eqref{eq:comp_Robin} does not hold true in general, if standard
finite elements are used for space discretization. Indeed, the stresses
will be discontinuous accross element boundaries and on polygonal
approximations of the boundary also $n_f$ will jump. It follows that 
the equality \eqref{eq:comp_Robin} can not be used and that a fully
discrete scheme based on the time-discrete approach of \cite{BDG19} has to
be carefully designed, with a discretization of the stresses that is
compatible with the loosely coupled scheme.

Drawing on ideas from \cite{le-tallec-mouro-01} (see also \cite[Algorithm 3]{burman2014explicit}) we consider a variational consistent representation 
of the interfacial fluid-stresses (i.e., as  the classical fluid variational residual involving a fluid-sided  lifting operator)
and show that the resulting scheme  can be recast as a Lagrange multiplier method. Matching the trace
spaces of the solid and fluid velocities, with that of the multiplier,
allows us to recover a relation similar to \eqref{eq:comp_Robin} in the fully
discrete framework. A complete a priori error analysis for the fully
discrete method, using piecewise affine approximation for all the
unknowns is then carried out resulting in an error estimate of
$\mathcal O \big(\sqrt{T} (\sqrt{\Delta t} + h)\big)$ in the natural norm. This shows
that our extension to the fully discrete case of the method proposed
in \cite{BDG19} is optimal, with no added conditions on the
discretization parameters or exponential growth of perturbations. To the best of our knowledge this is the first fully
discrete loosely coupled method for fluid-structure interaction
problems with thick-walled solids that allows for error estimates reflecting the splitting
error and the approximation order of the finite element space, without
any conditions on the physical or discretization parameters.
\subsection{Overview of previous work}
The source of instability occurring in loosely coupled methods was identified by Causin \emph{et al.} \cite{causin2005added} as
the so-called added-mass effect, see also \cite{le-tallec-mouro-01,FWR07}.
They also showed that the alternative, solving the interface coupling implicitly (strong coupling) and in a partitioned iterative fashion, on the
other hand is very costly in this regime, due to the stiffness of the coupling. 
A
first step in the direction of decoupling the two systems were the semi-implicit coupling schemes (see \cite{Fernandezetal2007,QuainiQuarteroni2007,BadiaQuainiQuarteroni2008b,ACF09,BukacCanicGlowMuhaQuaini2014}),
where the implicit part of the coupling, typically the elasticity
system and the fluid incompressibility (i.e., the added-mass), guarantees
stability, and the explicit step (transport in the fluid and geometrical non-linearities) reduces the
computational cost. Such splitting methods nevertheless retain an
implicit part, although of reduced size, and require a specific time-stepping in the fluid. Provably stable fully explicit coupling was first achieved
by Burman and Fern\'andez \cite{burman2009stabilization} using a formulation based on Nitsche's method,
drawing on an earlier, fully implicit formulation by Hansbo \emph{et al.}
\cite{Hansbo2005}. Stability was achieved by the addition of a temporal pressure stabilization that relaxed incompressibility in the vicinity
of the interface. Although the proposed scheme was proved to be
stable irrespectively of the added-mass effect, it suffered from a strong splitting error of order
$O(\Delta t/h)$  leading to a convergent scheme only for
$\Delta t = \mathcal O( h^{\alpha})$ with $\alpha>1$. The source of this
consistency error was the penalty term
of the Nitsche formulation. 
In a further development Burman and Fern\'andez compared the Nitsche
based method with a closely related scheme using a Robin type splitting
procedure \cite{burman2014explicit}. Robin type domain decomposition
had already been applied for the partitioned solution of strong coupling by Badia \emph{et
al.} \cite{badia2008fluid, NV08, GNV10} and Robin related explicit
coupling was proposed in \cite{BHS14}, but without theoretical justification. The loosely coupled scheme based on
Robin type coupling of \cite{burman2014explicit} was proved to be
stable, but with similar shortcomings as the Nitsche based
method. Since then several works \cite{fernandez2015generalized,BYZ15, BM16, FM16} have studied
the loosely coupled schemes for the interaction of
an incompressible fluid and a thick-walled solid. At best (see \cite{fernandez2015generalized,FM16}) their error
analysis results in 
estimates of order $O(\Delta t/\sqrt{h})$ under various (mild) conditions
on the discretization parameter. Observe that these latter references extend techniques
designed for the case of an incompressible fluid interacting with a
thin-walled solid (see \cite{Fern13,fernandez-mullaert-vidrascu-13}) to the case of the coupling with a thick-walled solid.   When preparing to submit the present work we came across a report
recently posted to arxiv by Seboldt and Buka\v c \cite{seboldt2020noniterative}, where a method using
Robin-conditions in a loosely coupled scheme
similar to the one introduced in [11] was analyzed. The main differences
in our work compared to theirs is that we use residual lifting, or Lagrange multipliers, for the interface
stresses and prove error estimates without conditions on the
discretization parameters and without exponential growth of the
stability constant in time. In their work on the other hand they derive a
stability estimate for the time semi-discretized problem where a
moving domain is accounted for and use arbitrary inf-sup stable finite
element spaces in the error analysis.

\subsection{Coupling of an incompressible fluid with a thin-walled solid}
Let us finally mention the case of an incompressible fluid coupled with
a thin-walled solid, i.e. a solid that is modeled on a domain of co-dimension
$1$ compared to the fluid system. This system is simpler and many
coupling schemes have been developed and analysed starting with the seminal work of 
Guidoboni \emph{et al.} \cite{GGCC09}, for instance \cite{Fern13,fernandez-mullaert-vidrascu-13,FLV15,BCM15,CMB15,OTB18}.

Since the
solid model is restricted to the $(d-1)$ dimensional interface domain
the solid velocities interact with the fluid everywhere in their
domain of definition. This means that there is no relaxation times
associated with propagation of waves in the direction perpendicular to
the interface. Therefore the stability of the solid system holds on the
$d-1$-dimensional interface and not in the $d$-dimensional bulk. 
The fundamental idea for stability is to implicitly integrate the solid inertial contributions 
within the fluid, through a Robin-type interface condition (which  
avoids the above mentioned added-mass issues) and  appropriately extrapolate 
the remaining solid contributions for accuracy (see, e.g., \cite{Fern13,fernandez-mullaert-vidrascu-13}).
Nevertheless,  when considering thick-walled solids, typically a trace inequality must be applied to
control interface quantities using the stability in the bulk domain (see, e.g., \cite{FM16}).
This leads to the need of control of higher derivatives, or a loss of a
negative power of the space mesh parameter. Therefore, when methods
used for the coupling with thin-walled solids are extended to the thick-walled solid case,  
 sub-optimal accuracy issues depending on the ratio of the time and space grid parameters result as in the examples in the previous section. 
 In other words, time splitting with thick-walled solids suffers from more severe accuracy issues 
 than in the thin-walled solid case.

\section{The linear fluid-solid interaction problem}

Let $\Omega_s$ and $\Omega_f$ be two  polygonal domains with a matching interface $\Sigma=\partial \Omega_s \cap \Omega_f$. For simplicity, we assume that the interface $\Sigma$ is a straight line. We also let $\Sigma_f= \Omega_f \backslash \Sigma$ and $\Sigma_s= \Omega_s \backslash \Sigma$. We consider the following coupled problem

\begin{equation}
\label{eq:fluid}
\left\{
\begin{aligned}
\rho_f \partial_t \bu-\dive \sigma_f(\bu,p)=&0 \quad && \text{ in } (0, T) \times \Omega_f, \\
\dive \bu = &0\quad && \text{ in } (0, T) \times \Omega_f, \\
\bu=&0  \quad && \text{ on }  (0, T) \times \Sigma_f,
\end{aligned}
\right.
\end{equation}

\begin{equation}
\label{eq:solid}
\left\{
\begin{aligned}
\rho_s \pt \bq- \dive \sigma_s(\bfeta)=&0 \quad && \text{ in } (0, T) \times \Omega_s, \\
\bq-\pt \bfeta=&0 \quad && \text{ in } (0, T) \times \Omega_s, \\
\bfeta=&0  \quad && \text{ on }  (0, T) \times \Sigma_s, \\
\end{aligned}
\right.
\end{equation}

\begin{equation}
\label{eq:coupling}
\left\{
\begin{aligned}
\bu=&\bq \quad  && \text{ on }  (0,T) \times \Sigma, \\
\sigma_{f}(\bu,p)\bn_f+  \sigma_{s}(\bfeta)  \bn_s= &0  \quad  &&  \text{ on } (0,T)  \times \Sigma,
\end{aligned}
\right.
\end{equation}
complemented with the following initial conditions: 
$$
\begin{aligned}
\bfeta(0, \cdot)=& \bfeta_0 \quad  && \text{ in } \Omega_s, \\
\bq(0,\cdot)=& \bq_0 \quad && \text{ in } \Omega_s, \\
\bu(0, \cdot)=& \bu_0 \quad  && \text{ in } \Omega_f.
\end{aligned}
$$
Here, $\bn_i$ is the outward pointing normal to $\partial \Omega_i$ for $i=s,f$. The stress tensors are given by
\begin{alignat*}{1}
\sigma_{f}(\bu,p) := & 2 \mu {\bf \epsilon} (\bfeta) - p {\bf I}, \\
\sigma_{s}(\bfeta) := & 2 L_{1} {\bf \epsilon}(\bfeta) + L_{2} (\dive \bfeta) {\bf I}.
\end{alignat*}
Here $\mu$ is the viscosity of the fluid and $L_1, L_2$ are the Lam\'e constants of the solid, with $L_1>0$ and $L_2 \ge 0$. The solid and fluid 
densities are denoted $\rho_s, \rho_f$, respectively.

Let us define the following spaces 
\begin{alignat*}{1}
\bV^s:=&\{ \bv \in \bH^1(\Omega_s): \bv =0 \text{ on } \Sigma_s \}, \\
\bV^f:=&\{ \bv \in \bH^1(\Omega_f): \bv =0\text{ on } \Sigma_f \}, \\
\bV^g:=&  L^2(\Sigma),\\
M^f:=& L_0^2(\Omega_f).
\end{alignat*}
We let 
\begin{equation*}
\blam:=\sigma_{f}(\bu,p)\bn_f.
\end{equation*}
Then if we assume that $\blam \in L^2(\Sigma)$ we have that the solution of \eqref{eq:fluid}-\eqref{eq:coupling} satisfies the weak formulation: For $t>0$, find $\bq(t), \bfeta(t) \in \bV^s$, $ u(t) \in \bV^f$, $ \blam(t) \in \bV^g$ , $p(t) \in M^f$ satisfying 
\begin{subequations}\label{weak}
\begin{alignat}{2}
\rho_s(\pt \bq, \bxi)_s+a_s( \bfeta, \bxi)+\la \blam, \bxi \ra=&0 \quad && \forall \bxi \in \bV^s \label{weak1}\\
(\bq, \bphi)_s-(\pt \bfeta, \bphi)_s=&0 \quad &&    \forall \bphi \in \bV^s \label{weak2} \\
\rho_f (\pt \bu, \bv)_f+a_f\big((\bu,p),(\bv,\theta)\big)  -\la \blam, \bv \ra=&0 \quad && \forall (\bv, \theta) \in \bV^f \times M^f  \label{weak3} \\
\la \bu-\bq, \bmu \ra=&0 \quad && \forall \bmu \in \bV^g. \label{weak4}
\end{alignat}
\end{subequations}
Here $(\cdot, \cdot)_i$ is the $L^2$ inner-product on $\Omega_i$, $i=s,f$. Also, $\la \cdot, \cdot \ra$ is the $L^2$ inner-product on $\Sigma$.   Finally, the bilinear form $a_f$ and $a_s$ are respectively given by  
$$
\begin{aligned}
a_f\big((\bu,p),(\bv,\theta)\big) := &2\mu (\eps(\bu), \eps(\bv))_f- (p, \dive \bv)_f + (\dive \bu, \theta),\\
a_s(\bfeta, \bxi):= &2L_1(\eps(\bfeta), \eps(\bxi))_s+ L_2 (\dive \bfeta, \dive \bxi)_s
\end{aligned}
$$
and the induced elastic energy norm 
\begin{equation*}
\|\bfeta\|_S^2:=a_s(\bfeta, \bfeta).
\end{equation*}

\section{Numerical method}

\subsection{Time discretization: Robin-based loosely coupled scheme}

We discretize the time interval $(0,T)$ with $N$ sub-intervals $(t_n, t_{n+1})$ where $t_n= \Dt n$,  $T=t_N$ and $\Delta t$ is the time-step length.  
We introduce the standard notation 
$$\pdt f^{n+1}:= \frac{1}{\Dt}(f^{n+1}-f^n),\quad f^{n+1/2}:=\frac{1}{2}(f^{n+1}+f^{n}).$$ 
As mentioned in the introduction, a splitting method was introduced in \cite{BDG19} using a Robin-based procedure that solves two PDEs, sequentially, in each time step. Here we further discretize that method by applying a backward Euler method in the fluid and a mid-point scheme  in the solid. 
This yields the  time semi-discrete solution procedure  reported in Algorithm~\ref{alg:semi}, where 
$\alpha>0$ denotes the so-called  Robin parameter (user defined). 
\begin{algorithm} [h!]
\noindent 
\begin{enumerate}
\item Solid subproblem:
\begin{equation}
\label{eq:solid-R}
\left\{
\begin{aligned}
\rho_s \pdt \bq^{n+1}- \dive \sigma_s(\bfeta^{n+\frac12})=&0 \quad && \text{ in }  \Omega_s, \\
\pdt \bfeta^n=&\bq^{n+\frac12} \quad && \text{ in }   \Omega_s, \\
\bfeta^{n+1}=&0  \quad && \text{ on }   \Sigma_s, \\
\sigma_s(\bfeta^{n+\frac12})n_s + \alpha \bq^{n+\frac12} =& \alpha \bu^n
- \sigma_f(\bu^n,p^n)n_f\quad && \text{ on }   \Sigma.
\end{aligned}
\right.
\end{equation}
\item Fluid subproblem:
\begin{equation}
\label{eq:fluid-R}
\left\{
\begin{aligned}
\rho_f \pdt \bu^{n+1}-\dive \sigma_f(\bu^{n+1},p^{n+1})=&0 \quad && \text{ in }  \Omega_f, \\
\dive \bu^{n+1} = &0\quad && \text{ in }  \Omega_f, \\
\bu^{n+1}=&0  \quad && \text{ on }   \Sigma_f,\\
\sigma_f(\bu^{n+1},p^{n+1})n_f+ \alpha \bu^{n+1}=& \alpha  \bq^{n+\frac12} +
\sigma_f(\bu^n,p^n)n_f\quad && \text{ on }   \Sigma.
\end{aligned}
\right.
\end{equation}
\end{enumerate}
\caption{Time semi-discrete, Robin-based, loosely coupled scheme (from \cite{BDG19}).}
\label{alg:semi}
\end{algorithm}
Note that Algorithm~\ref{alg:semi} is nothing but the generalization of the genuine Robin-Robin explicit coupling scheme 
introduced in \cite[Algorithm 4]{burman2014explicit} to the case of a general Robin coefficient $\alpha>0$  
(i.e., the traditional Nitsche penalty parameter $\gamma \mu/h$ is replaced by $\alpha$). 

Unconditional energy stability and  sub-optimal $\mathcal O(\sqrt{\Delta t})$ accuracy are derived in 
 \cite{BDG19} for the PDE version of Algorithm~\ref{alg:semi}, irrespectively of the value of $\alpha>0$. 
The relation \eqref{eq:fluid-R}$_4$ plays a fundamental role in the analysis of the method. 
The next section provides a fully discrete version of Algorithm~\ref{alg:semi} using a conforming 
finite element approximation in space. 

\subsection{Finite element approximation with fitted meshes}

We assume that $\mathcal{T}_h^i$ is a simplicial triangulation of $\Omega_i$ where $i=s,f$. We assume the meshes are quasi-uniform and shape-regular \cite{brenner2007mathematical}. Furthermore, we assume that the meshes match on the interface $\Sigma$. We define the following finite element spaces:
\begin{alignat*}{1}
\bV_h^s:=&\{ \bv \in \bV^s: \bv|_K \in  \bpol^1(K), \forall K \in \mathcal{T}_h^s \}, \\
\bV_h^f:=&\{ \bv \in \bV^f:  \bv|_K \in \bpol^1(K), \forall K \in \mathcal{T}_h^f \}, \\
\bV_h^g:=& \{ \text{ trace space of } \bV_h^f \text{ on } \Sigma \}, \\ 
M_h^f:=& \{ v \in M^f:  v|_K \in \pol^1(K), \forall K \in \mathcal{T}_h^f \}.
\end{alignat*}
Here $\pol^1(K)$ is the space of linear functions defined on $K$ and $\bpol^1(K)=[\pol^1(K)]^2$. 
Note that, owing to the mesh conformity, we have 
\begin{equation}\label{eq:traceS}
\mbox{trace}_{\vert\Sigma }\bV_h^s = \mbox{trace}_{\vert \Sigma} \bV_h^f = \bV_h^g.
\end{equation}
In order to circumvent the lack of inf-sup stability of the pair $\bV_h^f/M_h^f$, we consider the following pressure stabilized discrete bilinear 
form for the fluid (see, e.g., \cite{brezzi-pitkaranta-84}):
$$
a_{f,h}\big((\bu_h,p_h),(\bv_h,\theta_h)\big) :=  
a_{f}\big((\bu_h,p_h),(\bv_h,\theta_h)\big) + h^2 (\nabla p_h, \nabla \theta_h)_f.
$$
At last, we introduce the standard fluid-sided discrete lifting 
operator $\mathcal L_h: \bV_h^g \rightarrow \bV_h^f $, such that, 
the nodal values of ${\mathcal L}_h\bmu_h$ vanish out of $\Sigma$  and $ (\mathcal L_h \bmu_h) \vert_\Sigma = \bmu_h $, 
 for all $\bmu_h \in \bV_h^g$. 

\begin{algorithm} [h!]
\noindent 
\begin{enumerate}
\item Solid subproblem:
Find $\bq_h^{n+1}, \bfeta_h^{n+1} \in \bV_h^s$ such that $ \bq_h^{n+1/2} =  \pdt \bfeta_h^{n+1} $ and 
\begin{equation}  \label{solid1-R}
\rho_s(\pdt \bq_h^{n+1}, \bxi_h)_s+a_s( \bfeta_h^{n+1/2}, \bxi_h)+\alpha \la (\bq_h^{n+1/2} - \bu_h^n), \bxi_h\ra+ \la \blam_h^n, \bxi_h \ra=0 \quad  \forall \bxi_h \in \bV_h^s.
\end{equation}  
\item Fluid subproblem: Find $\bu_h^{n+1} \in \bV_h^f , p_h \in M_h^f $ such that
\begin{multline}\label{fluid-fd}
\rho_f (\pdt \bu_h^{n+1}, \bv_h)_f+a_{f,h}\big( (\bu_h^{n+1} ,p_h^{n+1}), (\bv_h,\theta_h)\big) \\
+\alpha \la \bu_h^{n+1}-\bq_h^{n+1/2} , \bv_h\ra -  \la  \blam_h^n,\bv_h\ra 
=0 \quad\forall (\bv_h, \theta_h) \in \bV_h^f \times M_h^f .
\end{multline}
\item Energy-preserving fluid-stress evaluation: Find $\blam_h^{n+1} \in \bV_h^g$ such that 
\begin{equation}\label{eq:stress}
\la \blam_h^{n+1}, \bmu_h\ra =  
 \rho_f (\pdt \bu_h^{n+1},{\mathcal L}_h\bmu_h)_f + a_{f,h}\big( (\bu_h^{n+1} ,p_h^{n+1}), ({\mathcal L}_h\bmu_h, 0 )\big)
\quad \forall \bmu_h \in \bV_h^g.
\end{equation}
\end{enumerate}
\caption{Fully discrete, Robin-based, loosely coupled scheme.}
\label{alg:fullRR}
\end{algorithm}

The proposed finite element approximation of Algorithm~\ref{alg:semi} is reported in 
Algorithm~\ref{alg:fullRR}. It should be noted that the interfacial fluid stress reconstruction given by step (3)
has been introduced for purely analysis purposes (see discussion below) and it should be omitted in any computer implementation. 
Indeed, there is no specific need of evaluating the Lagrange multiplier $\blam_h^{n+1}$ as an additional unknown, since the 
right-hand side of \eqref{eq:stress} can be inserted directly in \eqref{solid1-R} and \eqref{fluid-fd}.

Step (1) of Algorithm~\ref{alg:fullRR} can be reformulated as: 
Find $\bq_h^{n+1}, \bfeta_h^{n+1} \in \bV_h^s$ such that
\begin{subequations}\label{solid}
\begin{alignat}{2}
\rho_s(\pdt \bq_h^{n+1}, \bxi_h)_s+a_s( \bfeta_h^{n+1/2}, \bxi_h)+\alpha \la (\pdt \bfeta_h^{n+1}- \bu_h^n), \bxi_h\ra+ \la \blam_h^n, \bxi_h \ra=&0 \quad && \forall \bxi_h \in \bV_h^s, \label{solid1}\\
(\bq_h^{n+1/2}, \bphi_h)_s-(\pdt \bfeta_h^{n+1}, \bphi_h)_s=&0 \quad &&    \forall \bphi_h \in \bV_h^s. \label{solid2}
\end{alignat}
\end{subequations}
Moreover, from \eqref{fluid-fd} and \eqref{eq:stress}, we have that 
\begin{equation}\label{eq:stress2}
\la \blam_h^{n+1}, \bmu_h\ra =  \alpha \la \bq_h^{n+1/2}  -\bu_h^{n+1}, \bmu_h\ra +  \la  \blam_h^n,\bmu_h\ra 
\end{equation}
for all $\bmu_h \in \bV_h^g$.  As a result, steps (2) and (3) of Algorithm~\ref{alg:fullRR} can be reformulated as:
Find $\bu_h^{n+1} \in \bV_h^f , p_h \in M_h^f, \blam_h^{n+1} \in \bV_h^g$ such that
\begin{subequations}\label{fluid}
\begin{alignat}{2}
\rho_f (\pdt \bu_h^{n+1}, \bv_h)_f+2\mu (\eps(\bu_h^{n+1}), \eps( \bv_h))_f- (p_h^{n+1}, \dive \bv_h)_f -\la \blam_h^{n+1}, \bv_h \ra=&0 \quad && \forall \bv_h \in \bV_h^f \label{fluid1}\\
(\dive \bu_h^{n+1}, \theta_h)_f+ h^2 (\nabla p_h^{n+1}, \nabla \theta_h)_f=&0 \quad && \forall \theta \in M_h^f \label{fluid1.5} \\
\alpha \la \bu_h^{n+1}-\bq_h^{n+1/2} , \bmu_h \ra+ \la \blam_h^{n+1}-\blam_h^n, \mu \ra=&0 \quad && \forall \bmu_h \in \bV_h^g. \label{fluid2}
\end{alignat}
\end{subequations}

Finally, it is also worth noting that, owing to \eqref{eq:traceS} and \eqref{eq:stress2}, we have 
\begin{equation}\label{aux101}
\blam_h^{n+1}= \alpha(\bq_h^{n+1/2} -  \bu_h^{n+1}) + \blam_h^n \quad \text{ on } \Sigma 
\end{equation}
for $n\geq 1$. This relation, which represents discrete counterpart of 
\eqref{eq:fluid-R}$_4$, is a  fundamental ingredient of the stability analysis reported in next section.  Also, for convenience we write the relationship:
\begin{equation}\label{aux111}
\bq_h^{n+1/2}=\pdt \bfeta_h^{n+1} \quad \text{ on } \Omega_s.
\end{equation}

\subsection{Stability}
Next, we will prove stability of the method. The following identity is crucial.
\begin{equation}\label{eq131}
\bl \bv-\bw,  \bpsi\br= \frac{1}{2} \left(\|\bv\|_{L^2(\Sigma)}^2-\|\bw\|_{L^2(\Sigma)}^2+   \|\bpsi-\bw\|_{L^2(\Sigma)}^2-\|\bpsi-\bv\|_{L^2(\Sigma)}^2\right).
\end{equation}

The following quantities will allow us to state the stability result,
\begin{alignat*}{1}
\mathcal{S}_h^{n}=&\|\bfeta_h^{n}\|_{S}^2 +\rho_s\|\bq_h^{n}\|_{L^2(\Omega_s)}^2+ \rho_f \|\bu_h^{n}\|_{L^2(\Omega_f)}^2+ \Dt(\alpha \|\bu_h^n\|_{L^2(\Sigma)}^2+ \frac{1}{\alpha}   \|\blam_h^{n}\|_{L^2(\Sigma)}^2), \\
\mathcal{Z}_h^n=& \rho_f \|\bu_h^{n}-\bu_h^{n-1}\|_{L^2(\Om_f)}^2  +2\alpha \Dt \|\partial_{\Dt} \bfeta_h^{n}-\bu_h^{n-1}\|_{L^2(\Sigma)}^2 + 4 \mu \Dt  \|\eps(\bu_h^{n+1})\|_{L^2(\Omega_f)}^2 + 2h^2\Dt \|\nabla p_h^{n}\|_{L^2(\Omega_s)}^2.
\end{alignat*}

\begin{lemma}\label{lem:stab}
Let $\{ (\bq_h^{n+1}, \bfeta_h^{n+1},\bu_h^{n+1}, p_h^{n+1},\blam_h^{n+1} \}_{n=0}^{N-1} \subset   \bV_h^s\times \bV_h^s \times  \bV_h^f  \times M_h^f  \times \bV_h^g$ 
be given by Algorithm~\ref{alg:fullRR}. The following energy identity holds:
\begin{equation*}
\mathcal{S}_h^{M} +\sum_{m=1}^M  \mathcal{Z}_h^m=\mathcal{S}_h^{0} \qquad \text{ for }  1 \le M  \le N.
\end{equation*}
\end{lemma}

\begin{proof}
We let  $\bxi_h=\bq_h^{n+1/2}$ in \eqref{solid1} and using \eqref{aux111} we get
\begin{alignat*}{1}
\frac{1}{2}\|\bfeta_h^{n+1}\|_{S}^2+ \frac{\rho_s}{2}\|\bq_h^{n+1}\|_{L^2(\Omega_s)}^2=& \frac{1}{2} \| \bfeta_h^{n}\|_{S}^2+\frac{\rho_s}{2}  \|\bq_h^{n}\|_{L^2(\Omega_s)}^2.\\
&+\alpha \Dt \bl (\bu_h^{n} -\partial_{\Dt} \bfeta_h^{n+1}), \partial_{\Dt} \bfeta_h^{n+1} \ra  -\Dt \la \blam_h^{n}, \pdt \bfeta_h^{n+1} \ra .
\end{alignat*}
If we now set  $\bv_h=\bu_h^{n+1}$ in  \eqref{fluid1}  and $\theta_h=p_h^{n+1}$ in \eqref{fluid1.5} we obtain
\begin{alignat*}{1}
&\frac{\rho_f}{2} \|\bu_h^{n+1}\|_{L^2(\Omega_f)}^2+\frac{\rho_f}{2} \|\bu_h^{n+1}-\bu_h^{n}\|_{L^2(\Omega_f)}^2+ 2 \mu \Dt  \|\eps(\bu_h^{n+1})\|_{L^2(\Omega_f)}^2 +h^2\Dt \|\nabla p_h^{n+1}\|_{L^2(\Omega_f)}^2 \\
& =\frac{\rho_f}{2} \|\bu_h^{n}\|_{L^2(\Omega_f)}^2+ \Dt \la \blam_h^{n+1}, \bu_h^{n+1}\ra.
\end{alignat*}
Adding the above two equations we get
\begin{alignat*}{1}
&\frac{1}{2}\|\bfeta_h^{n+1}\|_{S}^2+ \frac{\rho_s}{2}\|\bq_h^{n+1}\|_{L^2(\Omega_s)}^2+ \frac{\rho_f}{2} \|\bu_h^{n+1}\|_{L^2(\Omega_f)}^2\\
&+\frac{\rho_f}{2} \|\bu_h^{n+1}-\bu_h^{n}\|_{L^2(\Omega_f)}^2+ 2 \mu \Dt  \|\eps(\bu_h^{n+1})\|_{L^2(\Omega_f)}^2 +h^2\Dt \|\nabla p_h^{n+1}\|_{L^2(\Omega_f)}^2 \\
&=\frac{1}{2}\|\bfeta_h^{n}\|_{S}^2+ \frac{\rho_s}{2}\|\bq_h^{n}\|_{L^2(\Omega_s)}^2+ \frac{\rho_f}{2} \|\bu_h^{n}\|_{L^2(\Omega_f)}^2  + \Dt J
\end{alignat*}
where 
\begin{alignat*}{1}
J:=& \alpha \bl (\bu_h^{n} -\partial_{\Dt} \bfeta_h^{n+1}), \partial_{\Dt} \bfeta_h^{n+1} \br-\bl \blam_h^{ n}, \partial_{\Dt} \bfeta_h^{n+1}  \br \\ 
&+\bl  \blam_h^{n+1}, \bu_h^{n+1} \br 
\end{alignat*}
After some manipulations and using \eqref{aux101} we obtain
\begin{alignat*}{1}
J=& \alpha \bl (\bu_h^{n} - \bu_h^{n+1} + \bu_h^{n+1} -\partial_{\Dt} \bfeta_h^{n+1}), \partial_{\Dt} \bfeta_h^{n+1} \br-\bl \blam_h^{ n}, \partial_{\Dt} \bfeta_h^{n+1} - \bu_h^{n+1} + \bu_h^{n+1}  \br \\ 
&+\bl  \blam_h^{n+1}, \bu_h^{n+1} \br \\
=&  \alpha \bl \bu_h^{n} - \bu_h^{n+1}, \partial_{\Dt} \bfeta_h^{n+1} \br + \alpha \bl  \bu_h^{n+1} -\partial_{\Dt} \bfeta_h^{n+1}, \partial_{\Dt} \bfeta_h^{n+1} \br \\
& - \bl \blam_h^{ n}, \partial_{\Dt} \bfeta_h^{n+1} - \bu_h^{n+1} \br + \bl \blam_h^{n+1} - \blam_h^{n}, \bu_h^{n+1} \br \\
=& \alpha \bl \bu_h^{n} - \bu_h^{n+1}, \partial_{\Dt} \bfeta_h^{n+1} \br -\frac{1}{\alpha} \bl \blam_h^{ n}, \blam_h^{n+1} - \blam_h^{n} \br \\
&  + \bl \blam_h^{n+1} - \blam_h^{n}, \bu_h^{n+1} - \partial_{\Dt} \bfeta_h^{n+1} \br.
\end{alignat*}
With \eqref{eq131}, we have  
\begin{alignat*}{1}
 \alpha \bl \bu_h^{n} - \bu_h^{n+1}, \partial_{\Dt} \bfeta_h^{n+1} \br =& \frac{\alpha}{2} \bigg(\|\bu_h^{n}\|_{L^2(\Sigma)}^2-\|\bu_h^{n+1}\|_{L^2(\Sigma)}^2\\
& - \|\partial_{\Dt} \bfeta_h^{n+1}-\bu_h^n\|_{L^2(\Sigma)}^2  + \|\partial_{\Dt} \bfeta_h^{n+1}-\bu_h^{n+1}\|_{L^2(\Sigma)}^2 \bigg), \\
\frac{1}{\alpha} \bl \blam_h^{ n}, \blam_h^{n+1} - \blam_h^{n} \br =& \frac{1}{2\alpha}\bigg(\|\blam_h^{n}\|_{L^{2}(\Sigma)}^2- \|\blam_h^{n+1}\|_{L^{2}(\Sigma)}^2 + \|\blam_h^{n+1} -  \blam_h^{ n} \|_{L^{2}(\Sigma)}^2 \bigg).
\end{alignat*}

Using \eqref{aux101}, we note that $ \|\blam_h^{n+1} -  \blam_h^{ n} \|_{L^{2}(\Sigma)}^2 = \alpha^2 \|\partial_{\Dt} \bfeta_h^{n+1}-\bu_h^{n+1}\|_{L^2(\Sigma)}^2$. Thus, we conclude

\begin{alignat*}{1}
J  =&\frac{\alpha}{2} \left(\|\bu_h^{n}\|_{L^2(\Sigma)}^2-\|\bu_h^{n+1}\|_{L^2(\Sigma)}^2\right) + \frac{1}{2\alpha} \left(\|\blam_h^{n}\|_{L^{2}(\Sigma)}^2- \|\blam_h^{n+1}\|_{L^{2}(\Sigma)}^2\right) -\frac{\alpha}{2}\|\partial_{\Dt} \bfeta_h^{n+1}-\bu_h^n\|_{L^2(\Sigma)}^2.
\end{alignat*}

We finally arrive at
\begin{alignat*}{1}
\frac{1}{2} \mathcal{S}_h^{n+1}+\frac{1}{2}\mathcal{Z}_h^{n+1}=& \frac{1}{2}\mathcal{S}_h^{n}.
\end{alignat*}
The result now follows after summing both sides. 
\end{proof}

\section{Error Analysis}
\subsection{The linear interpolant and the $L^2$- projection}
In order to carry out the error analysis we define the following discrete errors: 
\begin{alignat*}{2}
\bH_h^n:= &\bfeta_h^n -R_h^s \bfeta^n , \quad  &&\bQ_h^n:= \bq_h^n -R_h^s q^n, \\
\bU_h^n:=&\bu_h^n -R_h^f (\bu^n), \quad && \bLam_h^n:= \blam_h^n -\mathbb{P}_h \blam^n, \\
P_h^n:=& p_h^n -S_h (p^n),
\end{alignat*}
where $R_h^i$, $i=s,f$ are the Scott-Zhang interpolants defined in \cite{scott1990finite} projecting onto our finite element spaces $V_h^s$, $V_h^f$. Also,  $S_h$ is the Scott-Zhang interpolant into $M_h^f$ modified by a global constant so the average on $\Omega_f$ is zero. Finally, $\mathbb{P}_h$ is the $L^2$ projection onto $V_h^g$. Thus,
\begin{equation}\label{L2}
\bl \mathbb{P}_h \blam^n, \bmu_h \br = \bl \blam^n,\bmu_h \br, \quad \forall \bmu \in V_h^g.
\end{equation}
There is flexiblity in defining the Scott-Zhang interpolant and we choose the degrees of freedom on the boundary so that, $R_h^s \bv =R_h^f \bw$ on $\Sigma$ if $\bv \in [H^1(\Omega_s)]^2$ and $\bw \in [H^1(\Omega_f)]^2$ and  $\bv=\bw$ on $\Sigma$ . For these interpolants, we have the well known stability result for $\bv \in [H^1(\Omega_i)]^2$, $i=s,f$, and $r \in H^1(\Omega_f)$,
\begin{equation}\label{Stability}
\|R_h^i \bld{v} \|_{H^{1}(\Omega_i)} \leq C \|\bld{v} \|_{H^{1}(\Omega_i)}, \quad  \|S_h r \|_{H^{1}(\Omega_f)} \leq C \|r \|_{H^{1}(\Omega_f)}.
\end{equation}
We will also need the trace inequality
\begin{equation}\label{trace}
\|\bld{v}\|_{L^{2}(\Sigma)} \leq C\|\bld{v}\|_{H^{1}(\Omega_i)}.
\end{equation}

Furthermore, it is well known that for $\bld{v} \in [H^{2}(\Omega_i)]^2$, $i=s,f$, and $r \in  H^{2}(\Omega_f)$,
\begin{alignat}{3}
\|R_h^i \bld{v} - \bld{v}\|_{L^{2}(\Omega_{i})} &\leq C h^{2} \|\bld{v}\|_{H^{2}(\Omega_i)}, \quad & \|S_h r - r \|_{L^{2}(\Omega_{f})} &\leq C h^{2} \|r\|_{H^{2}(\Omega_f)}, \label{eq1} \\
\|R_h^i \bld{v} - \bld{v}\|_{H^{1}(\Omega_{i})} &\leq C h \|\bld{v}\|_{H^{2}(\Omega_i)},  \quad & \|S_h r - r \|_{H^{1}(\Omega_{f})} &\leq C h \|r\|_{H^{2}(\Omega_f)}.\label{eq2}
\end{alignat}

The interpolants restricted to the $\Sigma$ will be the Scott-Zhang interpolant on $\Sigma$ so we have 
\begin{alignat*}{3}
\|R_h^i \bld{v} - \bld{v}\|_{L^{2}(\Sigma)}+ h\|R_h^i \bld{v} - \bld{v}\|_{H^{1}(\Sigma)} &\leq C h^2 \|\bld{v}\|_{H^{2}(\Sigma)}.
\end{alignat*}
Thus, using the trace estimate with this approximation result we have 
\begin{alignat}{3}
\|R_h^i \bld{v} - \bld{v}\|_{L^{2}(\Sigma)}+ h\|R_h^i \bld{v} - \bld{v}\|_{H^{1}(\Sigma)} &\leq C h^2 \|\bld{v}\|_{H^{3}(\Omega_i)}, \quad \forall \bld{v} \in [H^3(\Omega_i)]^2. \label{eq2S} 
\end{alignat}



We may now state consistency-type results for the solid and the fluid.
\begin{lemma}\label{consistency2}
The following identities hold for all $\bxi_h \in \bV_h^s$, $\bv_h \in \bV_h^f$, and $\bmu_h \in \bV_h^g$.
\begin{alignat}{1}
\nonumber & \rho_s ( \pdt R_h^s \bq^{n+1}, \bxi_h)_s+a_s(R_h^s \bfeta^{n+1/2}, \bxi_h) +\alpha \la (\pdt R_h^s \bfeta^{n+1}- R_h^f \bu^n), \bxi_h\ra+ \la  \mathbb{P}_h \blam^n, \bxi_h \ra \\
 &= T_1 (\bxi_h) + \frac{1}{2}T_2 (\bxi_h) + V_1(\bxi_h)-  S_2 (\bxi_h) +  S_3 (\bxi_h) \label{cons11} ,
\end{alignat}
\begin{alignat}{2}
\rho_f(\pdt R_h^f \bu ^{n+1}, \bv_h)_f+2\mu (\eps( R_h^f \bu ^{n+1}), \eps( \bv_h))_f- (S_h p^{n+1},\dive \bv_h)_f-\la \mathbb{P}_h \blam^{n+1}, \bv_h \ra=&S_1(\bv_h) + V_2(\bv_h), && \label{cons22} \\
(\dive R_h^f \bu^{n+1}, \theta_h)_f+ h^2 (\nabla S_h p^{n+1}, \nabla \theta_h)_f=& V_3(\theta_h) + V_4(\theta_h), \quad &&  \label{cons33}
\end{alignat}
where
\begin{alignat*}{1}
T_1(\bxi_h):=&  \rho_s (\pdt R_h^s \bq^{n+1}-\pt  \bq^{n+1/2}, \bxi_h)_s ,\\
T_2(\bxi_h):=& \la  \blam^n-\blam^{n+1}, \bxi_h \ra , \\
S_1(\bv_h):= & \rho_f (\pdt R_h^f \bu^{n+1}- \pt R_h^f u^{n+1}, \bv_h)_f ,\\
S_2(\bmu_h) :=&\alpha \la R_h^f \bu^{n+1}-\pdt R_h^s \bfeta^{n+1}, \bmu_h \ra ,\\
S_3(\bmu_h):= & \alpha \la R_h^f \bu^{n+1}- R_h^f \bu^{n}, \bmu_h\ra ,\\
V_1(\bxi_h):=& a_{s}(R_h^s \bfeta^{n+1/2} - \bfeta^{n+1/2},\bxi_h)_s ,\\
V_2(\bv_h):=& 2\mu (\eps(R_h^f \bu^{n+1} - \bu^{n+1}), \eps(\bv_h))_f - (S_h p^{n+1} - p^{n+1}, \dive\bv_h)_f ,\\
V_3(\theta_h):=& (\dive (R_h^f \bu^{n+1} - \bu^{n+1}), \theta_h)_{f} , \\
V_4(\theta_h):=&  h^{2}( \nabla S_h p^{n+1}, \nabla \theta_{h})_{f}.
\end{alignat*}
\end{lemma}

\begin{proof}
For \eqref{cons11}, let $\mathbb{L}_1$ denote the left-hand side. Then we have
\begin{alignat*}{1}
\mathbb{L}_1=&\rho_s ( \pdt R_h^s \bq^{n+1} -  \pt \bq^{n+1/2}, \bxi_h)_s+ \rho_s ( \pt \bq^{n+1/2}, \bxi_h)_s + a_s(R_h^s \bfeta^{n+1/2},  \bxi_h) \\
&+\alpha \la (\pdt R_h^s \bfeta^{n+1}- R_h^f \bu^{n+1}), \bxi_h\ra + S_3(\bxi_h) + \la \blam^n, \bxi_h \ra \\
=&T_{1}(\bxi_h)  - a_{s}(\bfeta^{n+1/2},\bxi_h) + a_s(R_h^s \bfeta^{n+1/2},  \bxi_h) -\la \blam^{n+1/2},\bxi_h \ra \\
&-S_2(\bxi_h) + S_3(\bxi_h) + \la \blam^n, \bxi_h \ra  \\
=& T_{1}(\bxi_h) + V_1(\bxi_h) +\frac{1}{2} T_{2}(\bxi_h) -  S_2 (\bxi_h) +  S_3 (\bxi_h).
\end{alignat*}

For \eqref{cons22}, let $\mathbb{L}_2$ denote the left hand side. We have
\begin{alignat*}{1}
\mathbb{L}_2 =&  \rho_f   (\pdt R_h^f \bu^{n+1} - \pt \bu^{n+1}, \bv_h)_f + \rho_f  (\pt \bu^{n+1}, \bv_h)_f \\
&+2\mu (\eps(R_h^f \bu^{n+1}), \eps( \bv_h))_f- (S_h p^{n+1}, \dive \bv_h)_f-\la  \blam^{n+1}, \bv_h \ra\\
=& S_1(\bv_h) -2\mu (\eps(\bu^{n+1}), \eps( \bv_h))_f + ( p^{n+1}, \dive \bv_h)_f + \la \blam^{n+1}, \bv_h \ra \\
&+2\mu (\eps(R_h^f \bu^{n+1}), \eps( \bv_h))_f- (S_h p^{n+1}, \dive \bv_h)_f-\la  \blam^{n+1}, \bv_h \ra\\
=& S_1(\bv_h)  + V_2(\bv_h).
\end{alignat*}

Recall that $\dive \bu^{n} = 0$ for all $n\geq 1$.  For \eqref{cons33}, let $\mathbb{L}_3$ denote the left hand side. We have
\begin{alignat*}{1}
\mathbb{L}_3 =&(\dive  R_h^f \bu^{n+1}, \theta_h)_f + h^2 (\nabla S_h p^{n+1}, \nabla \theta_h)_f  \\
=& (\dive  R_h^f \bu^{n+1} - \dive \bu^{n+1}, \theta_h)_f + h^2 (\nabla S_h p^{n+1}, \nabla \theta_h)_f \\
=& V_3(\theta_h) + V_4(\theta_h).
\end{alignat*}

\end{proof}

From Lemma \ref{consistency2}, we find that the following error equations follow immediately after applying \eqref{solid} and \eqref{fluid}.
\begin{corollary}
For all $\bxi_h \in \bV_h^s$, $\bv_h \in \bV_h^f$, and  $\theta_h \in M_h^f$, the following identities hold.
\begin{alignat}{1}
\nonumber & \rho_s ( \pdt \bQ_h^{n+1}, \bxi_h)_s+a_s(\bH_h^{n+1/2}, \bxi_h) +\alpha \la (\pdt \bH_h^{n+1/2}- \bU_h^n), \bxi_h\ra+ \la  \bLam_h^n, \bxi_h \ra \\
 &=-T_1 (\bxi_h) - \frac{1}{2}T_2 (\bxi_h) - V_1(\bxi_h) +  S_2 (\bxi_h) -  S_3 (\bxi_h)\label{errsolid11} ,
\end{alignat}
\begin{alignat}{2}
\rho_f(\pdt \bU_h ^{n+1}, \bv_h)_f+2\mu (\eps(\bU_h^{n+1}), \eps( \bv_h))_f- (P_h^{n+1}, \dive \bv_h)_f-\la  \bLam_h^{n+1}, \bv_h \ra=&-S_1(\bv_h) - V_2(\bv_h), && \label{errfluid11}\\
(\dive \bU_h^{n+1}, \theta_h)_f+ h^2 (\nabla P_h^{n+1}, \nabla \theta_h)_f=&-V_3(\theta_h) - V_4(\theta_h). \quad &&  \label{errfluid22}
\end{alignat}
\end{corollary}

We also will need the following identities.
\begin{lemma}
The following identities hold:
\begin{subequations}\label{ids1}
\begin{equation}\label{erraux1011}
\bLam_h^{n+1}-\bLam_h^n= \alpha(\pdt \bH_h^{n+1}-  \bU_h^{n+1})+ \bg_1^{n+1}
\end{equation}
\begin{equation}\label{erraux1021}
\bQ_h^{n+1/2}=\pdt \bH_h^{n+1}- R_h^s \bg_2^{n+1},
\end{equation}
\end{subequations}
where
\begin{alignat*}{1}
\bg_1^{n+1}&:=\alpha(\pdt R_h^s \bfeta^{n+1}- R_h^f \bu^{n+1}) -(  \mathbb{P}_h \blam^{n+1}-\mathbb{P}_h \blam^n), \\
\bg_2^{n+1}&:= \bq^{n+1/2}-\pdt \bfeta^{n+1}.
\end{alignat*}
\end{lemma}
\begin{proof}

Using \eqref{aux101} we obtain
\begin{alignat*}{2}
\bLam_h^{n+1}-\bLam_h^n=& (\blam_h^{n+1}-\blam_h^n)- (\mathbb{P}_h \blam^{n+1}-\mathbb{P}_h\blam^n) \quad &&  \\
=&  \alpha(\pdt \bfeta_h^{n+1}-  \bu_h^{n+1}) - (\mathbb{P}_h \blam^{n+1}-\mathbb{P}_h \blam^n) \quad &&\\
=&  \alpha(\pdt \bH_h^{n+1}- \bU_h^{n+1})+\bg_1^{n+1}. \quad &&
\end{alignat*}
We also have by \eqref{solid2}
\begin{alignat*}{1}
\bq_h^{n+1/2}=&\bq_h^{n+1}-R_h^s  \bq^{n+1/2} \\
=& \pdt \bH_h^{n+1}+\pdt R_h^s \bfeta^{n+1} -R_h^s \bq^{n+1/2}  \\
=&\pdt \bH_h^{n+1}- R_h^s \bg_2^{n+1}.
\end{alignat*}
\end{proof}

\subsection{Approximation Results}
Finally, before we prove error estimates we will prove approximation inequalities. We recall the definition of space-time norms where $X$ is a Hilbert space
\begin{equation*}
\|\bv\|_{L^2(r_1,r_2 ; X)}^2:= \int_{r_1}^{r_2} \|\bv(\cdot, s)\|_X^2 ds.
\end{equation*}
The following are a series of approximation estimates. The proofs are elementary and appear in the appendix for completeness. 

\begin{lemma}\label{ineqLemma}
The following inequalities hold
\begin{subequations}
\begin{alignat}{2}
 \|\pdt R_h^s \bq^{n+1}-\pt \bq^{n+1/2}\|_{L^2(\Omega_s)}^2 \le& C \left(\frac{ h^{4}}{\Dt} \|\pt \bq\|_{L^2(t_n, t_{n+1}; H^2(\Omega_s))}^2 +\Dt^3  \|\pt^3 \bq\|_{L^2(t_n, t_{n+1}; L^2(\Omega_s))}^2\right), \label{tt1}\\ 
 \|\blam^{n+1}- \blam^n\|_{L^2(\Sigma)}^2\le &   C \Dt \bigg( \mu^2 \|\pt \bu\|_{L^2(t_n, t_{n+1}; H^2(\Omega_f))}^2 + \|\pt p\|_{L^2(t_n, t_{n+1}; H^1(\Omega_f))}^2\bigg) ,  \label{tt2} \\
  \|  R_h^f \bu^{n+1}- R_h^f \bu^{n}\|_{L^2(\Sigma)}^2 \le &  C \Dt   \|\pt \bu\|_{L^2(t_n, t_{n+1}; H^1(\Omega_f))}^2, \label{tt4} \\
    \|R_h^s \bg_2^{n+1}\|_{L^2(\Sigma)}^2 \le &  C \Dt^3   \|\pt^3 \bfeta \|_{L^2(t_n, t_{n+1}; H^1(\Omega_s))}^2, \label{tt3} \\
 \|\pdt R_h^f \bu^{n+1}- \pt \bu^{n+1}\|_{L^2(\Omega_f)}^2 \le &C  \bigg(\frac{h^{4}}{\Dt} \|\pt \bu\|_{L^2(t_n, t_{n+1}; H^2(\Omega_f))}^2 + \Dt\|\pt^2 \bu\|_{L^2(t_n, t_{n+1}; L^2(\Omega_f))}^2 \bigg) , \label{tt5}\\ 
 \|\nabla R_h^s \bg_2^{n+1}\|_{L^2(\Omega_s)}^2 \le& C \Dt^3 \|\pt^3 \bfeta\|_{L^2(t_n, t_{n+1}; H^1(\Omega_s))}^2, \label{tt6}\\
 \|\bg_1^{n+1}\|_{L^2(\Sigma)}^2 \le & C \Dt \bigg( \mu^2\|\pt \bu\|_{L^2(t_n, t_{n+1}; H^2(\Omega_f))}^2  + \|\pt p\|_{L^2(t_n, t_{n+1}; H^1(\Omega_f))}^2 \nonumber\\
 &\hspace{1cm} +\alpha^2 \|\pt^2 \bfeta \|_{L^2(t_n, t_{n+1}; H^1(\Omega_f))}^2 \bigg) , \label{tt7} \\
\|S_h p^{n+1} - p^{n+1}\|_{L^{2}(\Omega_f)}^2 \le& C \bigg(\Dt h^4 \|\pt p\|_{L^2(t_n, t_{n+1}; H^2(\Omega_f))}^2 + \frac{h^4}{\Dt}\|p\|_{L^2(t_n, t_{n+1}; H^2(\Omega_f))}^2\bigg), \label{tt12} \\
\|\eps(R_h^f \bu^{n+1} - \bu^{n+1}) \|_{L^{2}(\Omega_f)}^2\le& C \bigg(\Dt h^2\|\pt \bu\|_{L^2(t_n, t_{n+1}; H^2(\Omega_f))}^2 + \frac{h^2}{\Dt}\|\bu \|_{L^2(t_n, t_{n+1}; H^2(\Omega_f))}^2\bigg), \label{tt13} \\
\|\nabla S_h p^{n+1}\|_{L^{2}(\Omega_f)}^2 \le&C\bigg( \Dt \|\pt p\|_{L^{2}(t_n, t_{n+1} ;H^{1}(\Omega_f))}^2 +\frac{1}{\Dt}\| p\|_{L^{2}(t_n, t_{n+1} ;H^{1}(\Omega_f))}^2 \bigg), \label{tt11} \\
\|\bu^{n+1}\|_{H^{3}(\Omega_f)}^2 \le& C(\Dt \|\pt \bu\|_{L^{2}(t_n, t_{n+1};H^{3}(\Omega_f))}^2+ \frac{1}{\Dt} \|\bu\|_{L^{2}(t_n, t_{n+1};H^{3}(\Omega_f))}^2) ,\label{tt16}\\ 
\|\nabla (R_h^s \bfeta^{n+1/2} - \bfeta^{n+1/2}) \|_{L^{2}(\Omega_s)}^2 \le& C \bigg(h^2 \Dt \|\pt \bfeta\|_{L^2(t_n, t_{n+1}; H^2(\Omega_s))}^2+ \frac{h^2}{\Dt}  \|\bfeta\|_{L^2(t_n, t_{n+1}; H^2(\Omega_s))}^2\bigg), \label{tt9} \\
\|\nabla (R_h^s-I) \pdt \bfeta^{n+1}\|_{L^{2}(\Omega_s)}^2 \le& C \frac{h^2}{\Dt}\|\pt \bfeta\|_{L^2(t_{n}, t_{n+1}; H^2(\Omega_s))}^2. \label{tt15}
\end{alignat}
\end{subequations}
\end{lemma}


 \subsection{Main Theorem}
Now we can prove the main error estimate.
We define the following quantities:
\begin{alignat*}{1}
\mathcal{S}_h^n:= & \| \bH_h^{n}\|_{S}^2+ \rho_s\|\bQ_h^{n}\|_{L^2(\Omega_s)}^2 +  \rho_f \|\bU_h^{n}\|_{L^2(\Omega_f)}^2,\\
 \mathcal{E}_h^n:= & \Dt \alpha \|\bU_h^{n}\|_{L^2(\Sigma)}^2 +\frac{\Dt}{\alpha} \|\bLam_h^{n}\|_{L^2(\Sigma)}^2, \\
\mathcal{W}_h^n :=&\rho_f \|\bU_h^{n}-\bU_h^{n-1}\|_{L^2(\Omega_f)}^2+ 4\mu \Dt \|\eps(\bU_h^n)\|_{L^2(\Omega_f)}^2+ 2\Dt h^2 \|\nabla P_h^n\|_{L^2(\Omega_s)}^2, \\
\mathcal{Z}_h^n:=&  \Dt \alpha \|  \pdt \bH_h^{n} - \bU_h^{n-1}\|_{L^2(\Sigma)}^2. 
\end{alignat*}

\begin{theorem}\label{mainthm2}
Let $(\bu, \blam, \bfeta, \bq) $ be a regular enough solution of \eqref{weak} and let  $\{( \bfeta_h^n, \bq_h^n,\bu_h^n,  p_h^n,\blam_h^n)\}_{n=1}^N$  be given by Algorithm~\ref{alg:fullRR}. 
The following discrete error estimate holds: 
\begin{alignat*}{1}
 &\max_{1 \le  m \le N} ( \mathcal{S}_h^{m} + \mathcal{E}_h^{m})+  \sum_{m=1}^N  \left( \mathcal{W}_h^{m}  + \mathcal{Z}_h^{m} \right) \le  4(\mathcal{S}_h^{0} + \mathcal{E}_h^{0})+C Y \Psi,
 \end{alignat*}
where
 \begin{alignat*}{1}
 Y:=&\bigg(\frac{T}{\alpha}(\mu^2 + 1) + \alpha T \bigg) \Dt + \bigg(\frac{1}{\alpha}(\mu^2 +1) + \rho_f T + \alpha \bigg)\Dt^2 \\
&+ T \alpha \Dt^3 + \bigg( T(\rho_s + 1) + \alpha + 1 \bigg) \Dt^4 + \bigg( T + 1 + \alpha \bigg) h^2 \\
&+ \bigg(T(\rho_f + \rho_s ) + \frac{1}{\mu}\bigg)h^4 + (1+ \mu)h^2\Dt^2 + \frac{ h^4 \Dt^2}{\mu},
 \end{alignat*}
 and
\begin{alignat*}{1}
\Psi:=&  \|\pt \bfeta \|_{L^2(0, T; H^2(\Omega_s))}^2  +\|\pt^2 \bfeta \|_{L^2(0, T; H^2(\Omega_s))}^2+ \|\pt^3 \bfeta\|_{L^2(0, T; H^1(\Omega_s))}^2  + \|\pt^4 \bfeta \|_{L^2(0, T; L^2(\Omega_s))}^2 \\
&+\|\bu\|_{L^2(0, T; H^3(\Omega_f))}^2 +\|\pt \bu\|_{L^2(0, T; H^2(\Omega_f))}^2 
+  \|\pt^2 \bu\|_{L^2(0, T; L^2(\Omega_f))}^2  
+  \|p\|_{L^2(0, T; H^1(\Omega_f))}^2\\
&+ \|\pt p\|_{L^2(0, T; H^1(\Omega_f))}^2 + \|\bfeta\|_{L^\infty (0,T;H^2(\Omega_s))}^2.
 \end{alignat*}

\end{theorem}

\begin{proof}

Using  \eqref{erraux1021} we have
\begin{alignat*}{1}
 a_s(\bH_{h}^{n+1/2}, \bQ_{h}^{n+1/2} )=  \frac{1}{2 \Dt}\|\bH_h^{n+1}\|_{S}^2-\frac{1}{2 \Dt}\|\bH_h^{n}\|_{S}^2 -a_s(\bH_h^{n+1/2}, R_h^s \bg_2^{n+1}).
\end{alignat*}

If we let $\bxi_h=\bQ_h^{n+1/2}$  in \eqref{errsolid11} we obtain
\begin{alignat}{1}
&\frac{1}{2}\|\bH_h^{n+1}\|_{S}^2+ \frac{\rho_s}{2}\|\bQ_h^{n+1}\|_{L^2(\Omega_s)}^2 \label{5341} \\
=&\frac{1}{2} \|\bH_h^{n}\|_{S}^2+\frac{\rho_s}{2}  \|\bQ_h^{n}\|_{L^2(\Omega_s)}^2  +\Dt \, a_s(\bH_h^{n+1/2}, R_h^s\bg_2^{n+1})  -\Dt T_1(\bQ_h^{n+1/2})  \nonumber\\
&-\frac{\Dt}{2} T_2(\bQ_h^{n+1/2}) - \Dt V_1(\bQ_h^{n+1/2})  +\Dt S_2(\bQ_h^{n+1/2})  -\Dt S_3(\bQ_h^{n+1/2})+ J_1,  \nonumber
\end{alignat}
where 
\begin{alignat*}{1}
J_1 := &-\alpha \Dt \bl ( \partial_{\Dt} \bH_h^{n+1} - \bU_h^{n}),\bQ_h^{n+1/2} \ra-\Dt \la \bLam_h^{n}, \bQ_h^{n+1/2} \ra .
\end{alignat*}
We simplify $J_1$ by using \eqref{erraux1011}
\begin{alignat*}{1}
J_1 =&-\alpha \Dt \bl ( \partial_{\Dt} \bH_h^{n+1} - \bU_h^{n+1}),\bQ_h^{n+1/2} \ra+ \Dt \la \bLam_h^{n+1} -\bLam_h^{n}, \bQ_h^{n+1/2} \ra\\
& - \alpha \Dt \bl \bU_h^{n+1} - \bU_h^{n}, \bQ_h^{n+1/2}\ra - \Dt \la \bLam_h^{n+1}, \bQ_h^{n+1/2}\ra \\
=&\Dt \la  \bg_1^{n+1}, \bQ_h^{n+1/2} \ra- \alpha \Dt \bl \bU_h^{n+1} - \bU_h^{n}, \bQ_h^{n+1/2}\ra - \Dt \la \bLam_h^{n+1}, \bQ_h^{n+1/2}\ra.
\end{alignat*}
Therefore, if we plug this in to \eqref{5341} we have
\begin{alignat}{1}
\nonumber & \frac{1}{2}\|\bH_h^{n+1}\|_{S}^2+ \frac{\rho_s}{2}\|\bQ_h^{n+1}\|_{L^2(\Omega_s)}^2  \\
=&\frac{1}{2} \|\bH_h^{n}\|_{S}^2+\frac{\rho_s}{2}  \|\bQ_h^{n}\|_{L^2(\Omega_s)}^2  +\Dt a_s(\bH_h^{n+1/2}, R_h\bg_2^{n+1})  -\Dt T_1(\bQ_h^{n+1/2})  \nonumber\\
&-\frac{\Dt}{2} T_2(\bQ_h^{n+1/2}) - \Dt V_1(\bxi_h)   +\Dt S_2(\bQ_h^{n+1/2})  -\Dt S_3(\bQ_h^{n+1/2}) \ra \nonumber \\
& +\Dt \la  \bg_1^{n+1}, \bQ_h^{n+1/2} - \alpha \Dt \bl \bU_h^{n+1} - \bU_h^{n}, \bQ_h^{n+1/2}\ra - \Dt \la \bLam_h^{n+1}, \bQ_h^{n+1/2}\ra.  \label{errorS1}
\end{alignat}

If we now set  $\bv_h=\bU_h^{n+1}$ in  \eqref{errfluid11} and $\theta_h= P_h^{n+1}$ in \eqref{errfluid22}  we get 

\begin{alignat}{1}
&\frac{\rho_f}{2} \|\bU_h^{n+1}\|_{L^2(\Omega_f)}^2+\frac{\rho_f}{2} \|\bU_h^{n+1}-\bU_h^{n}\|_{L^2(\Omega_f)}^2 
+ \Dt 2\mu \|\eps(\bU_h^{n+1})\|_{L^2(\Omega_f)}^2+ \Dt h^2 \|\nabla P_h^{n+1}\|_{L^2(\Omega_s)}^2 \nonumber \\
=&\frac{1}{2} \|\bU_h^{n}\|_{L^2(\Omega_f)}^2+ \Dt \la \bLam_h^{n+1}, \bU_h^{n+1}\ra  - \Dt S_1(\bU_h^{n+1}) 
 - \Dt V_2(\bU_h) -\Dt V_3(P_h) - \Dt V_4(P_h). \label{errorF1}
\end{alignat}
If we use that 
\begin{equation*}
 -\Dt S_3(\bQ_h^{n+1/2})-\frac{\Dt}{2} T_2(\bQ_h^{n+1/2}) +\Dt \la  \bg_1^{n+1}, \bQ_h^{n+1/2} \ra=\frac{\Dt}{2} T_2(\bQ_h^{n+1/2})
\end{equation*}
and add \eqref{errorS1} and \eqref{errorF1}, we may write the following
\begin{alignat}{1}
 \frac{1}{2} \mathcal{S}_h^{n+1}+ \frac{1}{2} \mathcal{W}_h^{n+1}
 = & \frac{1}{2} \mathcal{S}_h^{n}+ K_1+\cdots+ K_9 + J_2, \label{9151}
\end{alignat}
where 
\begin{alignat*}{3}
K_1:=& -\Dt T_1(\bQ_h^{n+1/2}), \quad & K_2:= \frac{\Dt}{2} T_2(\bQ_h^{n+1/2}), &\qquad K_3: =-\Dt S_3(\bQ_h^{n+1/2}),  \\
K_4:=& - \Dt S_1(\bU_h^{n+1}) , \quad & K_5:=\Dt a_s( \bH_h^{n+1/2}, R_h^s \bg_2^{n+1}), &\qquad  K_6:= -\Dt V_3(P_h^{n+1}),\\
 K_7 :=& -\Dt V_4( P_h^{n+1}), \quad & K_8:= -\Dt V_1(\bQ_h^{n+1/2}),& \qquad  K_9 :=-\Dt V_2(\bU_{h}^{n+1}).
\end{alignat*}
and
\begin{alignat*}{1}
J_2 := &- \alpha \Dt \bl \bU_h^{n+1} - \bU_h^{n},\bQ_h^{n+1/2}\ra - \Dt \la \bLam_h^{n+1}, \bQ_h^{n+1/2}\ra + \Dt \la \bLam_h^{n+1}, \bU_h^{n+1}\ra.
\end{alignat*}

Using \eqref{erraux1021} we see that
\begin{alignat*}{1}
J_2 =&- \alpha \Dt \bl \bU_h^{n+1} - \bU_h^{n},\pdt \bH_h^{n+1}\ra  - \Dt \la \bLam_h^{n+1}, \pdt \bH_h^{n+1}-\bU_h^{n+1}\ra  \\
& + \alpha \Dt \bl \bU_h^{n+1} - \bU_h^{n}, R_h^s \bg_2^{n+1}\ra+ \Dt \la \bLam_h^{n+1}, R_h^s \bg_2^{n+1}\ra. 
\end{alignat*}
Using \eqref{erraux1011} gives
\begin{alignat*}{1}
J_2 =&- \alpha \Dt \bl \bU_h^{n+1} - \bU_h^{n},\pdt \bH_h^{n+1}\ra  - \frac{{\Dt}}{\alpha} \la \bLam_h^{n+1}, \bLam_h^{n+1}-\bLam_h^n\ra,  \\
& + \alpha \Dt \bl \bU_h^{n+1} - \bU_h^{n}, R_h^s \bg_2^{n+1}\ra+ \Dt \la \bLam_h^{n+1}, R_h^s \bg_2^{n+1}+\frac{\bg_{1}^{n+1}}{\alpha}\ra .
\end{alignat*}

We can then use \eqref{eq131} to get 
\begin{alignat*}{1}
- \alpha \Dt \bl \bU_h^{n+1} - \bU_h^{n},\pdt \bH_h^{n+1}\ra=&  -\frac{\alpha \Dt}{2}(\|\bU_h^{n+1}\|_{L^2(\Sigma)}^2  - \|\bU_h^{n}\|_{L^2(\Sigma)}^2) \\
& -\frac{\alpha \Dt}{2}( \|\pdt \bH_h^{n+1} - \bU_h^n \|_{L^2(\Sigma)}^2 - \|\pdt \bH_h^{n+1} - \bU_h^{n+1}\|_{L^2(\Sigma)}^2 ). \\
 - \frac{{\Dt}}{\alpha} \la \bLam_h^{n+1}, \bLam_h^{n+1}-\bLam_h^n\ra=& - \frac{\Dt}{2\alpha}(\|\bLam_{h}^{n+1}\|_{L^2(\Sigma)}^2  - \| \bLam_h^{n}\|_{L^2(\Sigma)}^2 + \|\bLam_h^{n+1} - \bLam_h^n \|_{L^2(\Sigma)}^2 ).
\end{alignat*}

Next, we note from \eqref{erraux1011} that we have
\begin{alignat*}{1}
\frac{1}{2\alpha} \| \bLam_h^{ n+1}-\bLam_h^{n}\|_{L^2(\Sigma)}^2 =&\frac{\alpha}{2} \| \partial_{\Dt} \bH_h^{n+1} - \bU_h^{n+1}\|_{L^2(\Sigma)}^2+  \frac{1}{2\alpha}\|\bg_{1}^{n+1}\|_{L^2(\Sigma)}^2 \\
&+\la \partial_{\Dt} \bH_h^{n+1} - \bU_h^{n+1},\bg_1^{n+1} \ra + \la \bU_h^{n+1}- \bU_h^{n},\bg_1^{n+1} \ra .
\end{alignat*}

Therefore, combining the above equations we have
\begin{alignat*}{1}
J_2 := & -\frac{\alpha \Dt}{2}(\|\bU_h^{n+1}\|_{L^2(\Sigma)}^2  - \|\bU_h^{n}\|_{L^2(\Sigma)}^2 + \|\pdt \bH_h^{n+1} - \bU_h^n \|_{L^2(\Sigma)}^2) \\
& - \frac{\Dt}{2\alpha}(\|\bLam_{h}^{n+1}\|_{L^2(\Sigma)}^2  - \| \bLam_h^{n}\|_{L^2(\Sigma)}^2) - \frac{\Dt}{2\alpha}\|\bg_{1}^{n+1}\|_{L^2(\Sigma)}^2 -\Dt \la \partial_{\Dt} \bH_h^{n+1} - \bU_h^{n},\bg_1^{n+1} \ra\\
&+ \Dt \la \bLam_h^{n+1}, R_h^s \bg_2^{n+1} +\frac{\bg_{1}^{n+1}}{\alpha}\ra +  \Dt \bl \bU_h^{n+1} - \bU_h^{n}, \alpha R_h^s \bg_2^{n+1} + \bg_1^{n+1}\ra.
\end{alignat*}

Substituting this into \eqref{9151} we arrive at
\begin{alignat}{1}
& \frac{1}{2}\mathcal{S}_h^{n+1} + \frac{1}{2}\mathcal{E}_h^{n+1} + \frac{1}{2}\mathcal{W}_h^{n+1}  + \frac{1}{2}\mathcal{Z}_h^{n+1}+  \frac{\Dt}{2\alpha}\|\bg_{1}^{n+1}\|_{L^2(\Sigma)}^2 \nonumber \\
&=  \frac{1}{2}\mathcal{S}_h^{n} + \frac{1}{2}\mathcal{E}_h^{n}+ \sum_{i=1}^{12} K_i ,\label{8611}
\end{alignat}
where 
\begin{alignat*}{1}
K_{10}:=&  -\Dt \la \partial_{\Dt} \bH_h^{n+1} - \bU_h^{n},\bg_1^{n+1} \ra, \quad K_{11}:= \Dt \la \bLam_h^{n+1}, R_h^s \bg_2^{n+1} +\frac{\bg_{1}^{n+1}}{\alpha}\ra , \, \\
K_{12}:=&\Dt \bl \bU_h^{n+1} - \bU_h^{n}, \alpha R_h^s \bg_2^{n+1} + \bg_1^{n+1}\ra. 
\end{alignat*}

Before proceeding to bound all the terms, we apply \eqref{erraux1021} to $K_8$ and obtain 
\begin{alignat*}{1}
K_{8} =& -\Dt a_s ((R_h^s - I) \bfeta^{n+1/2}, \pdt \bH_h^{n+1}) + \Dt a_s ((R_h^s - I) \bfeta^{n+1/2}, R_h^s \bg_{2}^{n+1}).
\end{alignat*}
One can verify the following discrete integration by parts 
 \begin{alignat*}{1}
 -\Dt a_s ((R_h^s - I) \bfeta^{n+1/2}, \pdt \bH_h^{n+1}) =&B^{n+1}+ \Dt a_{s}((R_h^s-I)\pdt \bfeta^{n+1}, \bH_h^{n+1/2} ).
 \end{alignat*}
 where
 \begin{equation*}
B^{n+1}:= a_s ((R_h^s-I)\bfeta^{n}, \bH_h^n)-a_s ((R_h^s-I)\bfeta^{n+1}, \bH_h^{n+1} ).
\end{equation*}

Thus, we have
\begin{alignat*}{1}
K_8 =& B^{n+1}+   \Dt a_{s}((R_h^s-I)\pdt \bfeta^{n+1}, \bH_h^{n+1/2} )+\Dt a_s ((R_h^s - I) \bfeta^{n+1/2},R_h^s \bg_{2}^{n+1}).
\end{alignat*}

Now we bound each $K_i$ for $1 \le i \le 12 $. The number $\delta>0$ would be chosen sufficiently small later. Using the Cauchy-Schwarz inequality we get 
\begin{alignat*}{2}
K_1 \le  \Dt  \|\pdt R_h^s \bq^{n+1}-\pt \bq^{n+1/2}\|_{L^2(\Omega_s)} \|\bQ_h^{n+1/2}\|_{L^2(\Omega_s)}.
\end{alignat*}
If we appy the geometric-arithmetic mean inequality we get 
\begin{alignat*}{2}
K_1 \le  \de \frac{\rho_s\Dt}{T} ( \|\bQ_h^{n+1}\|_{L^2(\Omega_s)}^2 + \|\bQ_h^{n}\|_{L^2(\Omega_s)}^2)+ C(\de) T\Dt \rho_s \|\pdt R_h^s \bq^{n+1}-\pt \bq^{n+1/2}\|_{L^2(\Omega_s)}^2.
\end{alignat*}

To bound $K_2$ we use \eqref{erraux1021}
\begin{alignat*}{2}
K_2=  -\frac{\Dt}{2} ( T_2(\pdt \bH_h^{n+1}-\bU_h^n)+ T_2(\bU_h^n)+ T_2(R_h^s \bg_2^{n+1}).
\end{alignat*}
Therefore, afer using  the Cauchy-Schwarz inequality we have 
\begin{alignat*}{2}
K_2 \le \frac{\Dt}{2}  \|\blam^{n+1}-\blam^n\|_{L^2(\Sigma)} ( \|\pdt \bH_h^{n+1}-\bU_h^n\|_{L^2(\Sigma)}+ \|\bU_h^n\|_{L^2(\Sigma)} + \|R_h^s \bg_2^{n+1}\|_{L^2(\Sigma)}).
\end{alignat*}
Hence, using the geometric-arithmetic mean inequality we see that
\begin{alignat*}{2}
K_2 \le&  \de \Big(\frac{\Dt^2 \alpha}{T}  \|\bU_h^n\|_{L^2(\Sigma)}^2+ \Dt \alpha  \|\pdt \bH_h^{n+1}-\bU_h^n\|_{L^2(\Sigma)}^2  \Big) \\
& +\frac{C(\de)}{\alpha}(\Dt+T) \|\blam^{n+1}-\blam^n\|_{L^2(\Sigma)}^2+ C(\de) \alpha \Dt \|R_h^s \bg_2^{n+1}\|_{L^2(\Sigma)}^2.
\end{alignat*}
Similarly,  we have 
\begin{alignat*}{1}
K_3 \le \Dt  \alpha \|  R_h^f \bu^{n+1}- R_h^f \bu^{n}\|_{L^2(\Sigma)} ( \|\pdt \bH_h^{n+1}-\bU_h^n\|_{L^2(\Sigma)}+ \|\bU_h^n\|_{L^2(\Sigma)} + \|R_h^s \bg_2^{n+1}\|_{L^2(\Sigma)}),
\end{alignat*}
and
\begin{alignat*}{2}
K_3 \le& \de \Big(\frac{\Dt^2 \alpha}{T}  \|\bU_h^n\|_{L^2(\Sigma)}^2+ \Dt \alpha  \|\pdt \bH_h^{n+1}-\bU_h^n\|_{L^2(\Sigma)}^2  \Big) \\
& +C(\delta) \alpha(\Dt+T)\|  R_h^f \bu^{n+1}- R_h^f\bu^{n}\|_{L^2(\Sigma)} + C(\delta) \alpha\Dt \|R_h^s \bg_2^{n+1}\|_{L^2(\Sigma)}^2.
\end{alignat*}

Following this same process, we have
\begin{alignat*}{1}
K_4 &\le \de \frac{\rho_f\Dt}{T}\|\bU_h^{n+1}\|_{L^2(\Omega_f)}^2+  C(\de) \rho_f \Dt T \|\pdt R_h^f \bu^{n+1}- \pt \bu^{n+1}\|_{L^2(\Omega_f)}^2, \\
K_5 &\le  \de \frac{\Dt}{T}( \| \bH_h^{n+1}\|_{S}^2 +\| \bH_h^{n}\|_{S}^2)+ C(\de) T\Dt  \|R_h^s \bg_2^{n+1}\|_{S}^2, \\
K_7 & \le \de \Dt h^{2}\| \nabla P_h ^{n+1}\|_{L^2(\Omega_f)}^2 +C(\de) \Dt h^{2} \|\nabla S_h p^{n+1} \|_{L^2(\Omega_f)}^2 ,\\
K_8 &\le  \de \frac{\Dt}{T} (\| \bH_h^{n}\|_{S}^2+  \| \bH_h^{n+1}\|_{S}^2)+  C \Dt \|R_h^s \bg_2^{n+1}\|_{S}^2 \\
&\quad+C(\de) \Dt T \|(R_h^s-I)\pdt \bfeta^{n+1}\|_{S}^2 + C \Dt \|(R_h^s - I) \bfeta^{n+1/2}\|_{S}^2+B^{n+1}, \\
 K_{10} &\le \de \alpha \Dt \|\pdt \bH_h^{n+1} - \bU_h^{n} \|_{L^2(\Sigma)}^2 + \frac{C(\de)\Dt}{\alpha} \|\bg_{1}^{n+1}\|_{L^2(\Sigma)}^2 , \\
K_{11} &\le \de  \frac{(\Dt)^{2}}{T \alpha}\|\bLam_h^{n+1}\|_{L^2(\Sigma)}^2 + \frac{C(\de)T}{\alpha}\|\alpha R_h^s \bg_2^{n+1} + \bg_{1}^{n+1}\|_{L^2(\Sigma)}^2 ,\\
K_{12} & \le \de \frac{(\Dt)^2 \alpha}{T} (\|\bU_h^{n+1} \|_{L^2(\Sigma)}^2 + \| \bU_h^{n} \|_{L^2(\Sigma)}^2) + \frac{C(\de) T}{\alpha}\|\alpha R_h^s \bg_2^{n+1} + \bg_{1}^{n+1}\|_{L^2(\Sigma)}^2 .
\end{alignat*}

To estimate $K_6$, we perform integration by parts and proceed as before. Thus, 
\begin{alignat*}{1}
K_6 &= \Dt (R_h^f \bu^{n+1} - \bu^{n+1}, \nabla P_h^{n+1})_f - \Dt \bl (R_h^f \bu^{n+1} - \bu^{n+1})\cdot \bn, P_h^{n+1}\br \\
&\leq \Dt \|R_h^f \bu^{n+1} - \bu^{n+1}\|_{L^2(\Omega_f)}\|\nabla P_h\|_{L^2(\Omega_f)} + \Dt \|R_h^f \bu^{n+1} - \bu^{n+1}\|_{L^2(\Sigma)}\|P_h^{n+1}\|_{L^2(\Sigma)} \\
&\leq C\Dt h^2\|\bu^{n+1}\|_{H^2(\Omega_f)}\|\nabla P_h^{n+1}\|_{L^2(\Omega_f)} + C\Dt h^2\|\bu^{n+1}\|_{H^3(\Omega_f)}\|\nabla P_h^{n+1}\|_{L^2(\Omega_f)},
\end{alignat*}
where the last step follows from applying \eqref{eq2S} and using the trace inequality \eqref{trace} on $P_h^{n+1}$. We also used Poincare's inequality. Thus, applying this result along with Young's inequality, we have 
\begin{alignat*}{1}
K_6 &\leq  \de \Dt h^2\|\nabla P_h^{n+1}\|_{L^2(\Omega_f)}^2+ C(\de) \Dt h^2 \|\bu^{n+1}\|_{H^3(\Omega_f)}^2.
\end{alignat*}

Finally, for $K_{9}$, we can easily show that
\begin{alignat*}{1}
K_{9} \le & \de \Dt\mu \|\eps(\bU_h^{n+1})\|_{L^2(\Omega_f)}^2 +  C(\de) \Dt \mu \|\eps(R_h^f \bu^{n+1} - \bu^{n+1})\|_{L^2(\Omega_f)}^2 +  C(\de) \frac{\Dt}{\mu} \|S_h p^{n+1} - p^{n+1}\|_{L^2(\Omega_f)}^2.
\end{alignat*}

Combining the above inequalities, we have
\begin{alignat*}{1}
    \sum_{1\le i \le12} K_i \le  &  12 \delta \frac{\Dt}{T} (\mathcal{S}_h^{n+1}+\mathcal{S}_h^n+\mathcal{E}_h^{n+1}+\mathcal{E}_h^{n} ) +12 \de(\mathcal{Z}_h^{n+1}+\mathcal{W}_h^{n+1})+ C(\de) G^{n+1} + B^{n+1},
\end{alignat*}
where $G^{n+1}:=\sum_{i=1}^{14}  G_i^{n+1}$ such that
\begin{alignat*}{2}
G_1^{n+1} :=& T \Dt \rho_s  \|\pdt R_h^s \bq^{n+1}-\pt \bq^{n+1/2}\|_{L^2(\Omega_s)}^2, \quad &&  G_2^{n+1}:=  \frac{1}{\alpha}(\Dt + T)  \|\blam^{n+1}-\blam^n\|_{L^2(\Sigma)}^2,  \\
G_3^{n+1}:=& \alpha(\Dt + T) \|  R_h^f \bu^{n+1}- R_h^f \bu^{n}\|_{L^2(\Sigma)}^2, \quad &&  G_4^{n+1}:=  \Dt\alpha \|R_h^s \bg_2^{n+1}\|_{L^2(\Sigma)}^2,\\
G_5^{n+1}:= & \rho_f \Dt T \|\pdt R_h^f \bu^{n+1}- \pt \bu^{n+1}\|_{L^2(\Omega_f)}^2,   \quad && G_6^{n+1}:=  (\Dt + T\Dt)  \|R_h^s \bg_2^{n+1}\|_{S}^2, \quad\\
G_7^{n+1}:=& \frac{\Dt}{\alpha} \|\bg_1^{n+1}\|_{L^2(\Sigma)}^2,  \quad && G_8^{n+1}:= \frac{T}{\alpha} \|  \alpha  R_h^s  \bg_2^{n+1} +\bg_{1}^{n+1}  \|_{L^2(\Sigma)}^2, \\
G_9^{n+1} :=& \frac{\Dt}{\mu}\|S_h p^{n+1} - p^{n+1} \|_{L^2(\Omega_f)}^2,  \quad && G_{10}^{n+1}:=  \mu \Dt \|\eps(R_h^f \bu^{n+1} - \bu^{n+1}) \|_{L^2(\Omega_f)}^2, \\
G_{11}^{n+1}:=& \Dt h^{2} \|\nabla S_h p^{n+1}\|_{L^2(\Omega_f)}^2,  \quad && G_{12}^{n+1}:= \Dt h^2 \|\bu^{n+1}\|_{H^3(\Omega_f)}^2, \\
G_{13}^{n+1}:= & \Dt\|(R_h^s - I) \bfeta^{n+1/2}\|_{S}^2,  \quad && G_{14}^{n+1}:=  \Dt T \|(R_h^s - I)\pdt \bfeta^{n+1}\|_{S}^2.
\end{alignat*}

Therefore, if we take the sum of \eqref{8611} from $1$ to $M \leq N$, we have
\begin{alignat}{1}
& \frac{1}{2}(\mathcal{S}_h^{M} + \mathcal{E}_h^{M}) + \frac{1}{2}\sum_{m=1}^{M}\bigg(\mathcal{W}_h^{m}  +\mathcal{Z}_h^{m}\bigg) \nonumber \\
\le&  \frac{1}{2}(\mathcal{S}_h^{0} + \mathcal{E}_h^{0}) + 24 \de \max_{0 \le  m \le N} ( \frac{1}{2}\mathcal{S}_h^{m} + \frac{1}{2}\mathcal{E}_h^{m}) + 12 \de \sum_{m=1}^{M}\bigg( \mathcal{W}_h^{m} + \mathcal{Z}_h^{m}\bigg)  + \sum_{m=1}^M B^m  +C(\de) \sum_{m=1}^{M} G^{m}. \label{almostComplete}
\end{alignat}
We can use the telescoping sum to get
\begin{alignat*}{1}
\sum_{m=1}^{M} B^{m} &= a_s ((R_h^s-I)\bfeta^{0}, \bH_h^0)- a_s ((R_h^s-I)\bfeta^{M}, \bH_h^M).
\end{alignat*}

We may then bound this term using Cauchy-Schwartz and Young's inequality,  giving us
\begin{alignat*}{1}
\sum_{m=1}^{M} B^{m} \leq& \de ( \|\bH_h^M\|_S^2 + \|\bH_h^0 \|_S^2 ) + C(\de) (\|(R_h - I)\bfeta^{M}\|_{S}^2 + \|(R_h - I)\bfeta^{0}\|_{S}^2).
\end{alignat*}

If we take $\de$ small enough,  say $24 \de \le 1/2$, we obtain after using \eqref{almostComplete}
\begin{alignat}{1}
& \frac{1}{4}\max_{1 \le m \le N} (\mathcal{S}_h^{m} + \mathcal{E}_h^{m}) + \frac{1}{4}\sum_{m=1}^{N}\bigg(\mathcal{W}_h^{m}  +\mathcal{Z}_h^{m}\bigg) \nonumber \\
\le&  \mathcal{S}_h^{0} + \mathcal{E}_h^{0} +    C\max_{0 \le  m \le N} \|(R_h - I)\bfeta^{m}\|_{S}^2 + C\sum_{m=1}^{N} G^{m}.
\label{CompleteComplete}
\end{alignat}

 Now we proceed to bound  $\sum_{m=1}^N G_i^{m}$  for every $1 \le i \le 14$. Using \eqref{tt1} we have
 \begin{alignat*}{1}
 \sum_{m=1}^N G_1^{m} \le C T \rho_s   \left(h^{4} \|\pt \bq\|_{L^2(0, T; H^2(\Omega_s))}^2 +\Dt^4  \|\pt^3 \bq\|_{L^2(0, T; L^2(\Omega_s))}^2\right).
 \end{alignat*}

Using \eqref{tt2} we get 
  \begin{alignat*}{1}
  \sum_{m=1}^N G_2^{m}\le C  \left(\frac{1}{\alpha}(\Dt+T)\right) \Dt \bigg( \mu^2 \|\pt \bu\|_{L^2(0,T; H^2(\Omega_f))}^2 + \|\pt p\|_{L^2(0,T; H^1(\Omega_f))}^2\bigg) .
  \end{alignat*}

  If we apply \eqref{tt4} we obtain 
  \begin{alignat*}{1}
  \sum_{m=1}^N G_3^{m}  \le& C  \alpha (\Dt+T) \Dt \|\pt \bu\|_{L^2(0, T; H^1(\Omega_f))}^2 .
  \end{alignat*} 

From \eqref{tt3}, it follows that
\begin{alignat*}{1}
  \sum_{m=1}^N G_4^{m}\le C \Dt^4 \alpha   \|\pt^3 \bfeta \|_{L^2(0,T; H^1(\Omega_s))}^2.
\end{alignat*}

 We can use \eqref{tt5} to obtain
 \begin{alignat*}{1}
  \sum_{m=1}^N G_5^{m} \le  C T \rho_f \left(   h^{4} \|\pt \bu\|_{L^2(0, T; H^2(\Omega_f))}^2 + \Dt^2  \|\pt^2 \bu\|_{L^2(0, T; L^2(\Omega_f))}^2  \right).
 \end{alignat*}

As a result of \eqref{tt6} and \eqref{tt7} we have
 \begin{alignat*}{1}
  \sum_{m=1}^N G_6^{m} &\le C  (1+T) \Dt^4 \|\pt^3 \bfeta \|_{L^2(0, T; H^2(\Omega_s))}^2,\\
\sum_{m=1}^N G_7^{m} &\le C \frac{\Dt^2}{\alpha} \bigg( \mu^2\|\pt \bu\|_{L^2(0,T; H^2(\Omega_f))}^2  + \|\pt p\|_{L^2(0,T; H^1(\Omega_f))}^2 + \alpha^2 \|\pt^2 \bfeta \|_{L^2(0,T; H^1(\Omega_f))}^2 \bigg),\\
\sum_{m=1}^N G_8^{m} & \le C T \Dt^3\alpha \|\pt^3 \bfeta \|_{L^2(0, T; H^1(\Omega_s))}^2 \\
&\quad +C\frac{\Dt T}{\alpha} \bigg(  \mu^2\|\pt \bu\|_{L^2(0, T; H^2(\Omega_f))}^2 +\|\pt p\|_{L^2(t_n, t_{n+1}; H^1(\Omega_f))}^2  + \alpha^2 \|\pt^2 \bfeta \|_{L^2(0, T; H^1(\Omega_f))}^2 \bigg).
\end{alignat*}

Proceeding in the same manner, from \eqref{tt12} - \eqref{tt15} we have
\begin{alignat*}{1}
 \sum_{m=1}^N G_9^{m} &\le C \frac{h^4}{\mu} \bigg(\Dt^2\|\pt p\|_{L^2(0,T; H^2(\Omega_f))}^2 + h^4\|p\|_{L^2(0,T; H^2(\Omega_f))}^2\bigg),\\
\sum_{m=1}^N G_{10}^{m}&\le C \mu \bigg(\Dt^2 h^2\|\pt \bu\|_{L^2(0,T; H^2(\Omega_f))}^2 + h^2\|\bu \|_{L^2(0,T; H^2(\Omega_f))}^2\bigg),\\
\sum_{m=1}^N G_{11}^{m} &\le C h^2\bigg(\Dt^2 \|\pt p\|_{L^2(0,T; H^2(\Omega_f))}^2 + \|p\|_{L^2(0,T; H^2(\Omega_f))}^2\bigg),\\
\sum_{m=1}^N G_{12}^{m} &\le Ch^2(\Dt^2 + 1) \|\bu^{n+1}\|_{L^{2}(0,T;H^{3}(\Omega_f))}^2, \\
\sum_{m=1}^N G_{13}^{m} &\le C h^2 \bigg( \Dt^2 \|\pt \bfeta\|_{L^2(0,T; H^2(\Omega_s))}^2+  \| \bfeta\|_{L^2(0,T; H^2(\Omega_s))}^2\bigg) , \\
\sum_{m=1}^N G_{14}^m & \le C h^2 T\|\pt \bfeta\|_{L^2(0,T; H^2(\Omega_s))}^2.
\end{alignat*}

We can also have  the bound 
 \begin{equation*}
  \max_{0 \le  m \le N} \|(R_h - I)\bfeta^{m}\|_{S}^2 \le C  h^2  \max_{0 \le  m \le N} \|\bfeta(t_m)\|_{H^2(\Omega_s)}^2 \le C h^2 \|\bfeta\|_{L^\infty(0,T;H^2(\Omega_s))}^2. 
\end{equation*}

Thus, combining the terms we get 
\begin{equation*}
\max_{0 \le  m \le N} \|(R_h - I)\bfeta^{m}\|_{S}^2 + \sum_{m=1}^{N} G^{m} \le C Y \Psi. 
\end{equation*}
Plugging this into \eqref{CompleteComplete} completes the proof. 

\end{proof}

\section{Numerical experiments}

The purpose of this section is to illustrate, via numerical experiments, the performance of the loosely coupled scheme given by Algorithm~\ref{alg:fullRR}. 
We consider the well-known pressure wave propagation example (see, e.g., \cite[Section~6.1.1]{fernandez-mullaert-vidrascu-13b}). In \eqref{eq:fluid}-\eqref{eq:coupling}, we have   $\Omega_{f} = [0,L]\times [0,R]$,   
 $\Om_{s} = [0,L]\times [R,R+\epsilon]$, $\Sigma = [0,L]\times \{R\}$, $L=6$, $R=0.5$ and $\epsilon = 0.1 $. All the units are given in the CGS system. 
 At the left fluid boundary $x=0$ we impose a sinusoidal pressure of maximal amplitude 
 $ 2\times 10^4$ during $5\times 10^{-3}$ s,   corresponding to half a period. Free traction is enforced at $x=L$ 
 and a symmetry condition on the bottom wall. 
Transverse membrane effects in the solid are included through a zeroth-order term $c_0 \dep$ in \eqref{eq:solid}$_1$. Zero displacement  and zero traction are respectively
 enforced on the solid later and upper boundaries.
 The fluid physical parameters are  $ \rho^{\rm f}=1$ and $\mu=0.035$. For the solid we have 
 $ \rho^{\rm s}= 1.1$,  $L_1=1.15 \cdot 10^6$, $L_2=1.7 \cdot 10^6$ and $c_0 = 4 \cdot 10^6$. 
A multiplying coefficient of  $10^{-3}/\mu$ is applied to the Brezzi-Pitk\"aranta pressure stabilization method. 
 All the simulations have been performed with \texttt{FreeFem++} (see \cite{freefem}).

 \begin{figure}[!h]
 \centering
 \includegraphics[width=0.8\linewidth]{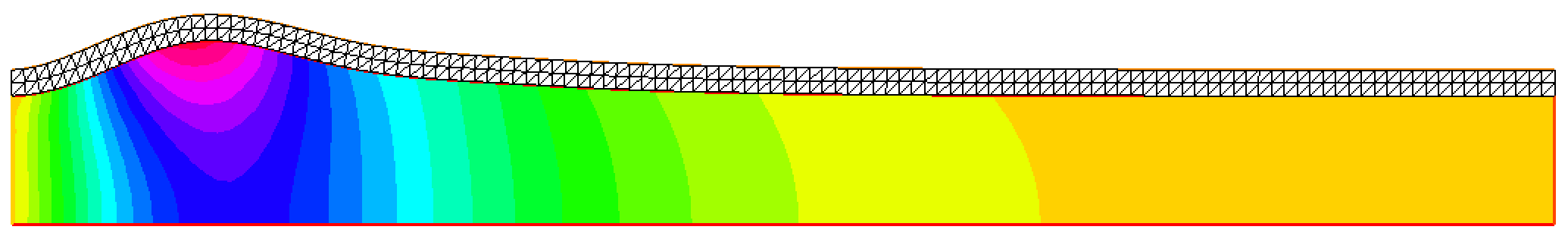}\\ 
 \includegraphics[width=0.8\linewidth]{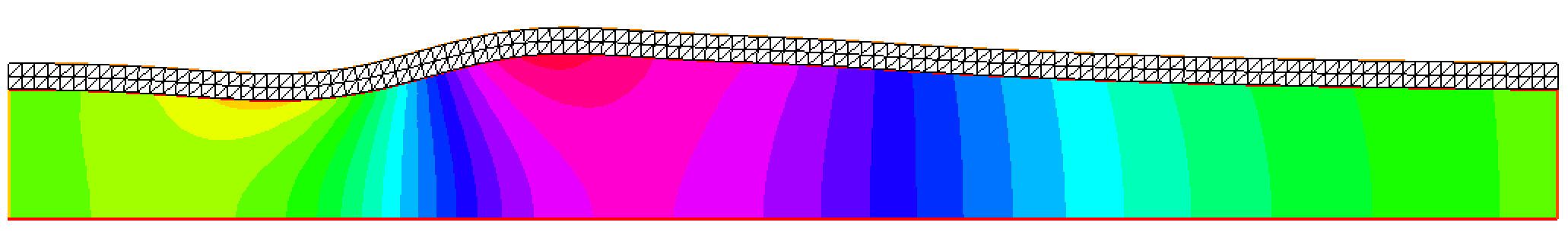}\\ 
 \includegraphics[width=0.8\linewidth]{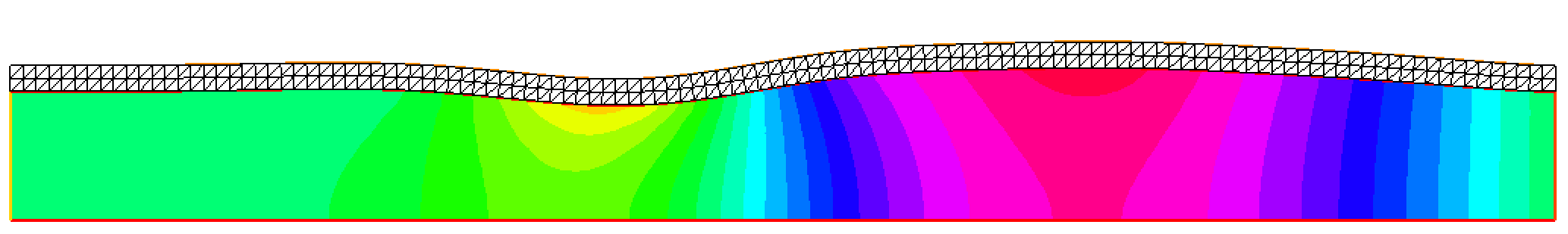}
 \caption{Snapshots of the fluid pressure and solid deformation at $t=5\cdot 10^{-3}$, $10^{-2}$ and $1.5\cdot 10^{-2}$ (from top to bottom). 
 Algorithm \ref{alg:fullRR} with $\tau = 2.5\cdot 10^{-4}$, $h = 0.05$ and $\alpha = 500$.}
 \label{fig:snap}
 \end{figure}

Figure \ref{fig:snap} shows some snapshots of the fluid pressure approximation obtained with  Algorithm~\ref{alg:fullRR} for
 $\tau = 2.5\cdot 10^{-4}$, $h = 0.05$ and $\alpha = 500$. For illustration purposes, the fluid and solid domains are 
 displayed in deformed configuration (magnified by a factor 5). The numerical solution remains 
 stable,  in agreement with Lemma~\ref{lem:stab},  and shows a propagating pressure-wave.

\subsection{Accuracy}

In order to asses the accuracy of Algorithm~\ref{alg:fullRR}, a reference solution has been generated using a strongly coupled scheme and 
 a high space-time grid resolution ($h = 3.125 \cdot 10^{-3}$, $\Delta t = 10^{-6}$).  Convergence histories are measured in terms 
of  the relative elastic energy-norm $\|\bfeta_{ref}^N -  \bfeta_h^N\|_S$ at time $t=0.015$, 
 by refining both in time and in space at the same rate, namely,  by taking 
\begin{equation*}\label{eq:parameter1}
	(\Delta t,h)\in  \left\{\left( \frac{5\cdot 10^{-4}}{2^i}, \frac{10^{-1}}{2^i}\right) \right\}_{i=0}^4  .
\end{equation*}
This allows, in particular, to highlight the $h$-uniformity of the error estimate provided in Theorem~\ref{mainthm2}.

Figure \ref{fig:err-time} reports the corresponding convergence histories obtained with Algorithm~\ref{alg:fullRR} with $\alpha=500$ and the strongly coupled scheme. We can clearly see that Algorithm~\ref{alg:fullRR} delivers an overall sub-optimal convergence rate, close to $\mathcal O(\sqrt{h})$. This
 is in agreement with the error estimate provided by Theorem~\ref{mainthm2} with $\Delta t = \mathcal O(h)$. The 
 strongly coupled scheme yields an overall $\mathcal O({h})$  accuracy, as expected.

 \begin{figure}[h!]
	\centering
\includegraphics[width=0.6\linewidth]{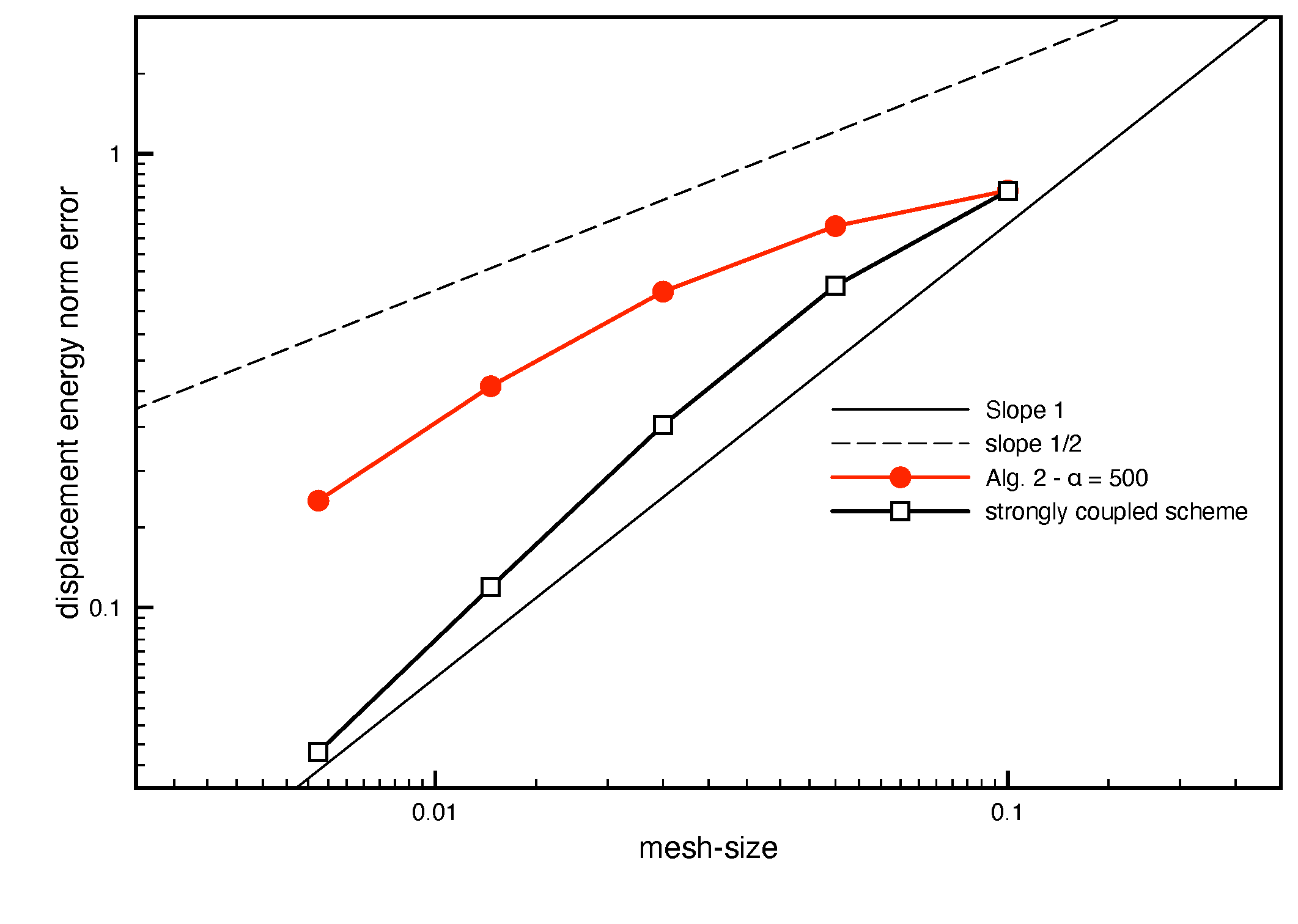} 
\caption{Time-convergence history of the displacement at $t=0.015$, with $\Delta t = \mathcal O (h)$ obtained with Algorithm~\ref{alg:fullRR}  ($\alpha=500$) and the strongly coupled scheme.}
\label{fig:err-time}
\end{figure}
\begin{figure}[h!]
	\centering
\includegraphics[width=0.6\linewidth]{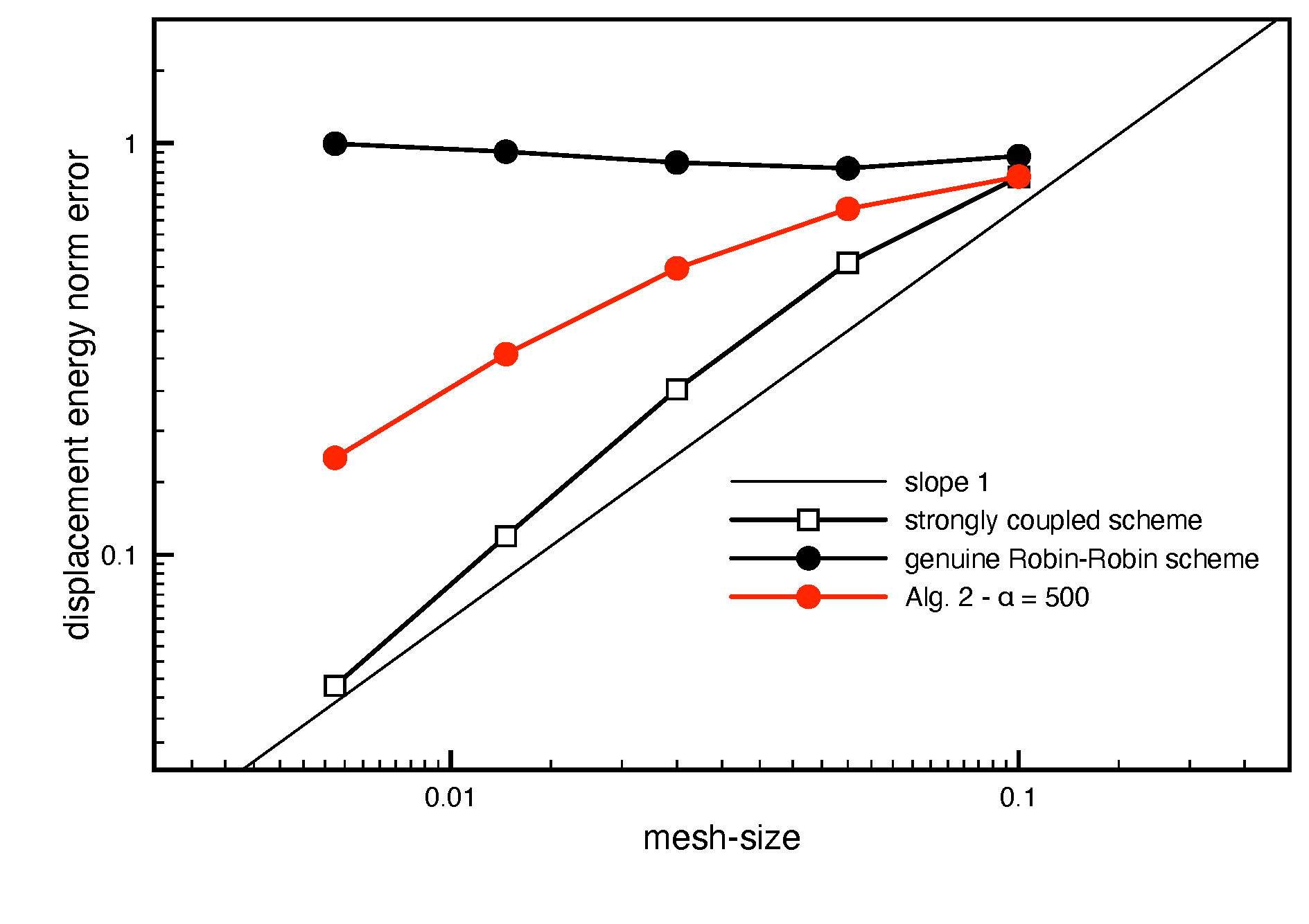} 
\caption{Time-convergence history of the displacement at $t=0.015$, with $\Delta t = \mathcal O (h)$ obtained with Algorithm~\ref{alg:fullRR}  ($\alpha=500$), the strongly coupled scheme and the 
genuine Robin-Robin explicit coupling scheme from  \cite[Algorithm 4]{burman2014explicit}.}
\label{fig:comp-genuine}
\end{figure}
 \begin{figure}[h!]
	\centering
\includegraphics[width=0.6\linewidth]{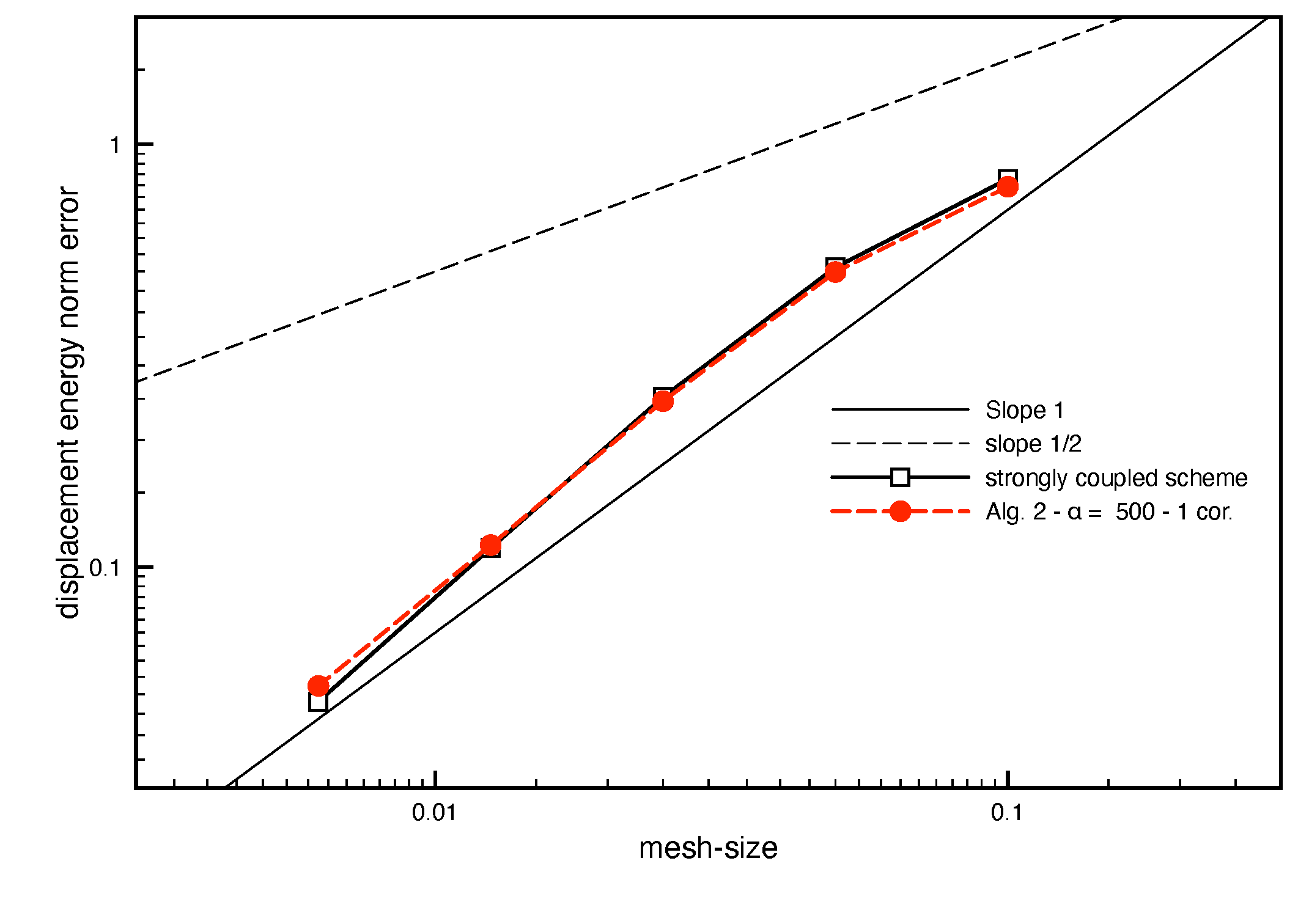} 
\caption{Time-convergence history of the displacement at $t=0.015$, with $\Delta t = \mathcal O (h)$ obtained with  the strongly coupled scheme
and Algorithm~\ref{alg:fullRR} with 1 correction iteration ($\alpha=500$).}
\label{fig:err-time-1c}
\end{figure}
\begin{figure}[h!]
	\centering
\subfigure[$\Delta t = 5\cdot 10^{-4}$, $ h= 0.1$.]{\includegraphics[width=0.48\linewidth]{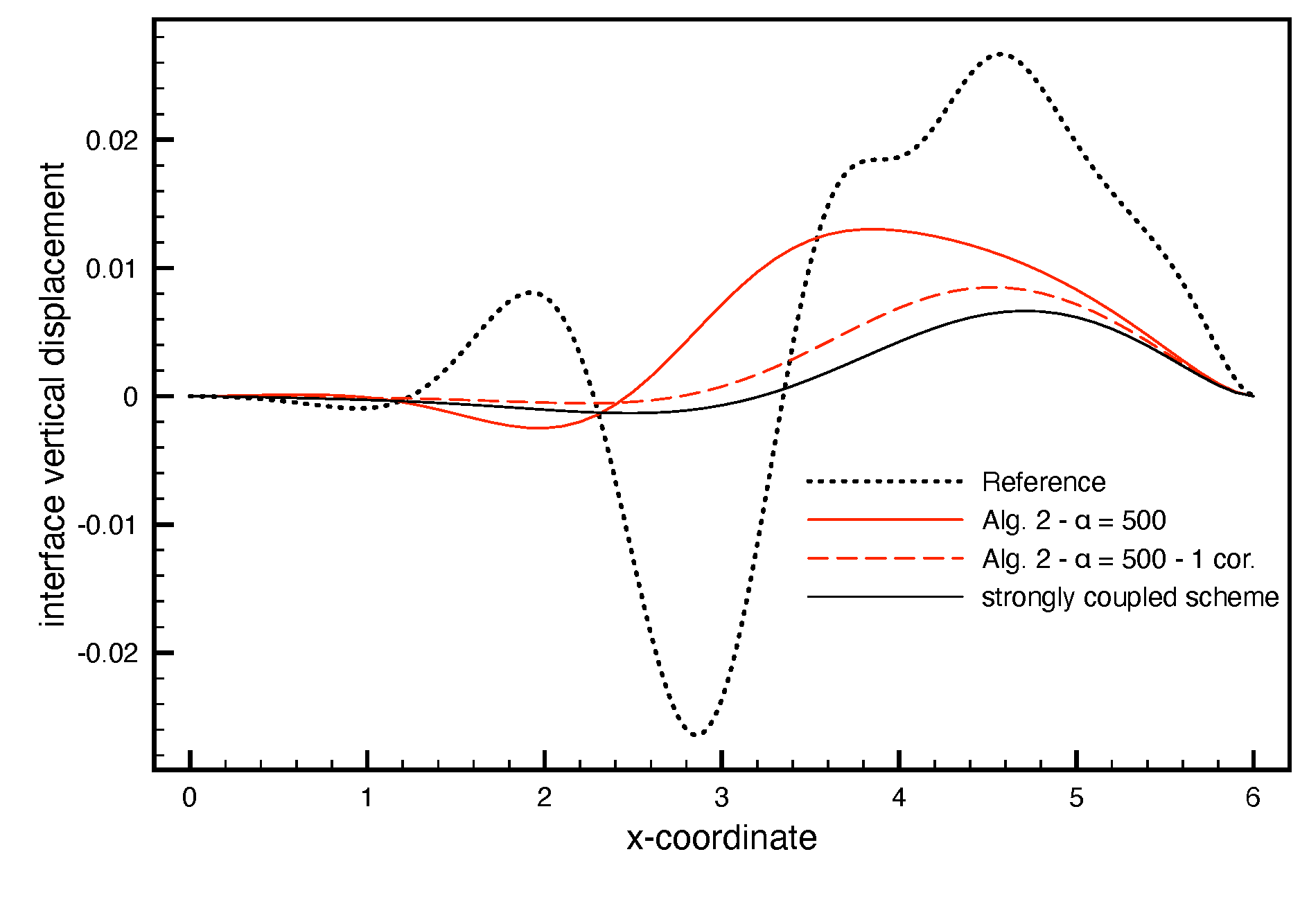} }\,\, 
\subfigure[$\Delta t = 2.5\cdot 10^{-4}$, $ h= 0.05$.]{\includegraphics[width=0.48\linewidth]{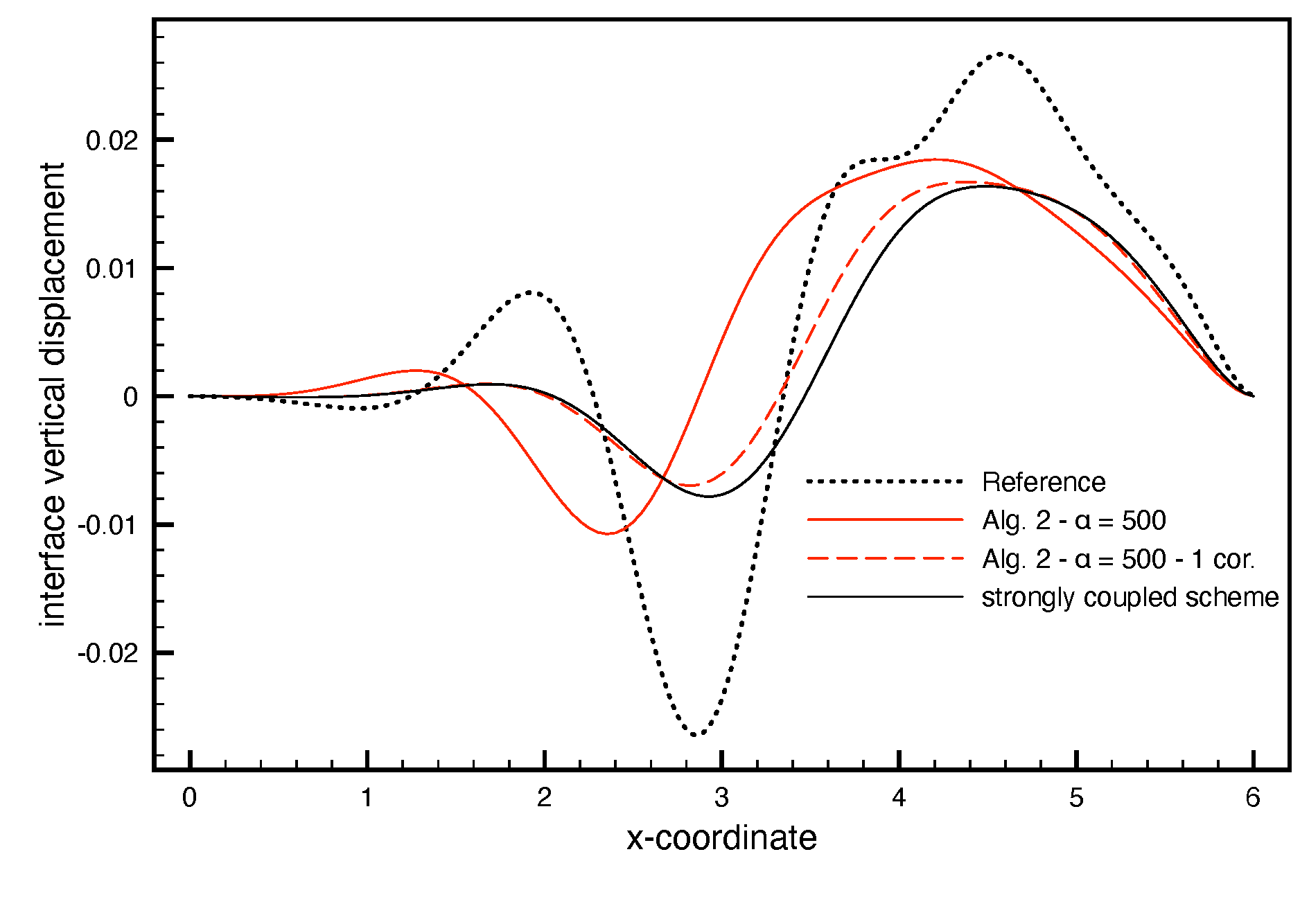} \label{fig:subfig3}}
\centering
\subfigure[$\Delta t = 1.25\cdot 10^{-4}$, $ h= 0.025$.]{\includegraphics[width=0.48\linewidth]{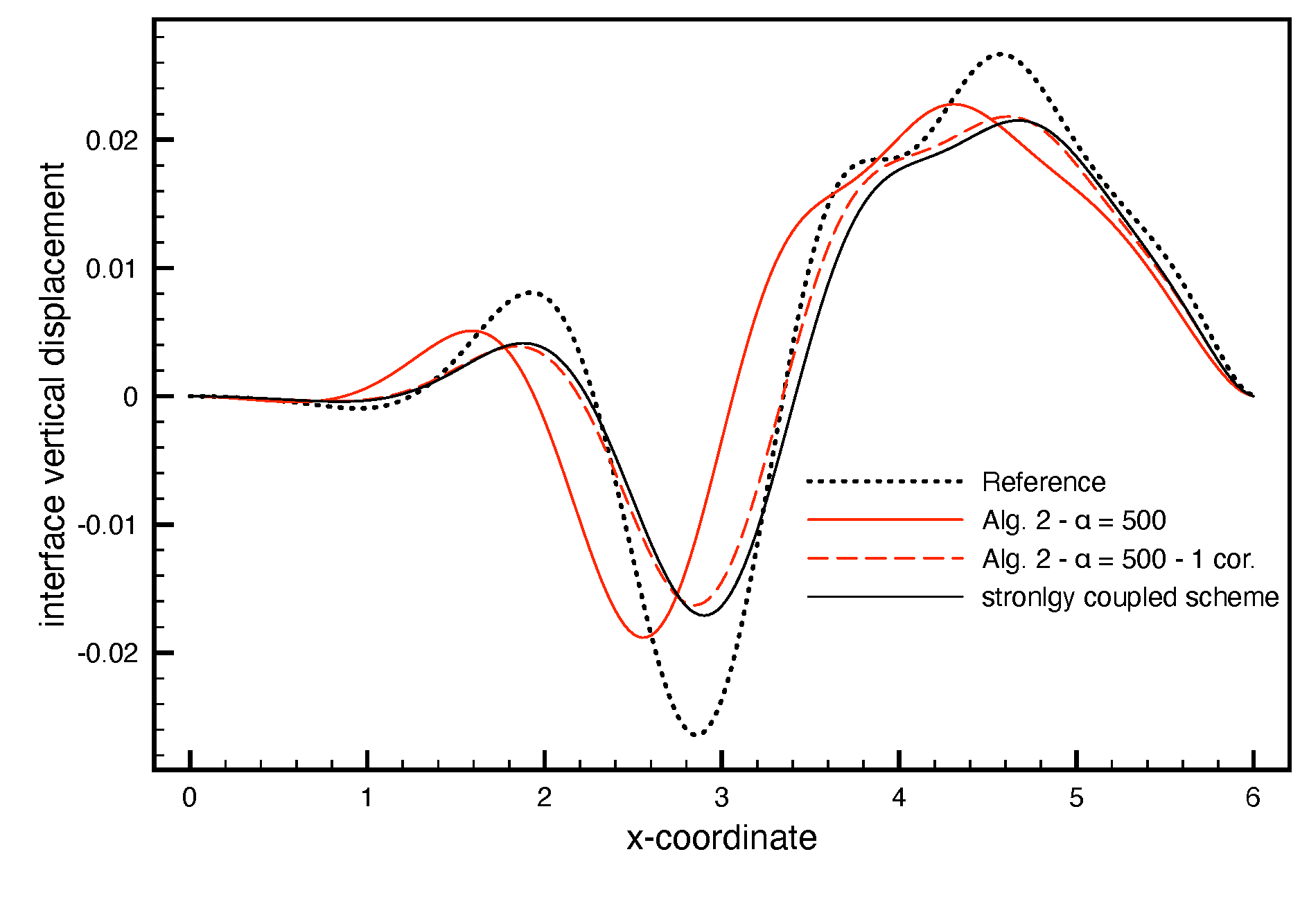}}\,\,\,
\subfigure[$\Delta t = 6.25\cdot 10^{-5}$, $ h= 0.0125$.]{\includegraphics[width=0.48\linewidth]{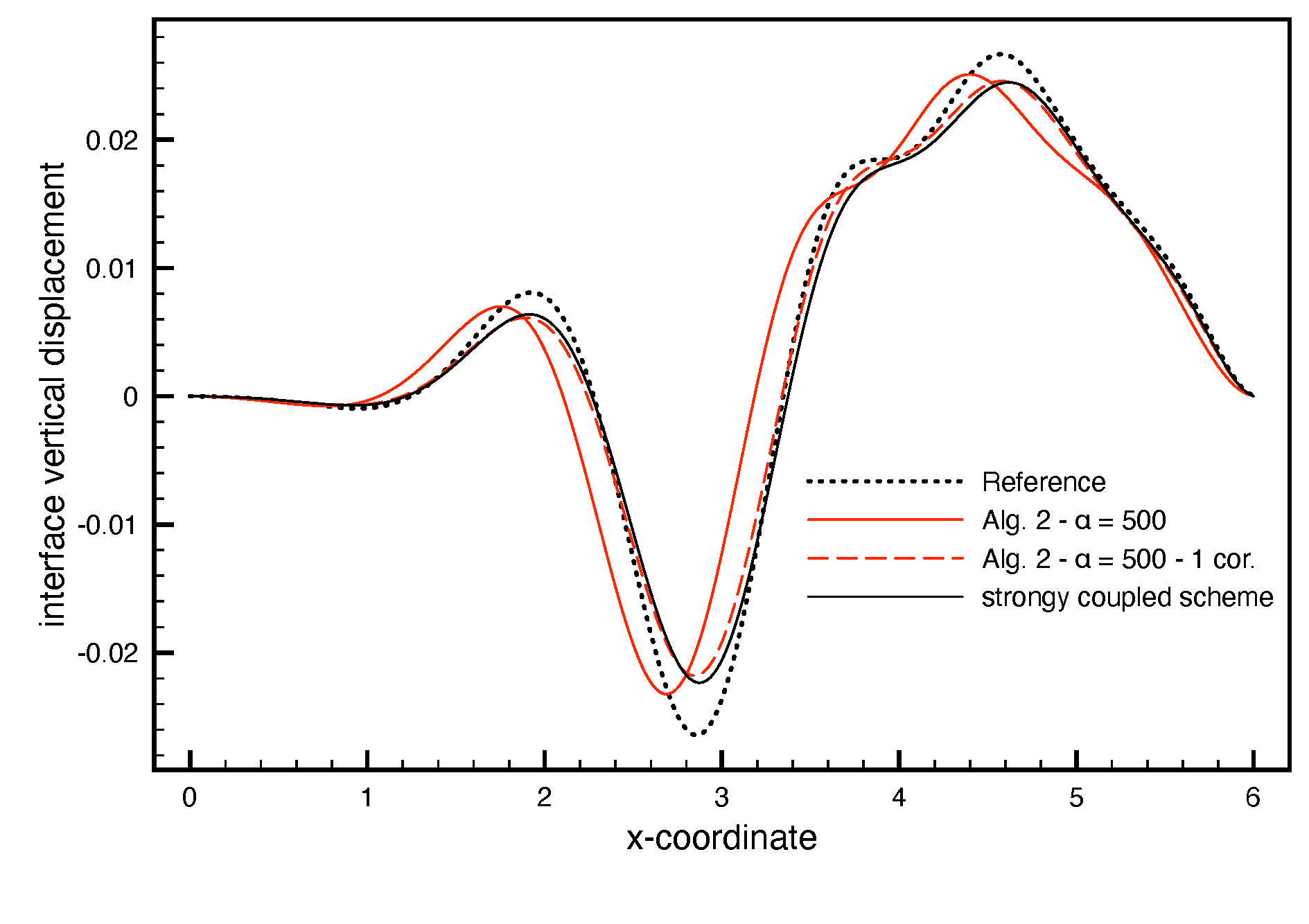} \label{fig:subfig2}}  
\caption{Comparison of the displacements at $t=0.015$ obtained for different levels of space-time refinement, $\Delta t = \mathcal O (h)$.}
\label{fig:disp}
\end{figure}

Another salient feature of Figure \ref{fig:err-time} is that it highlights the $h$-uniformity of the time-splitting error.
This is indeed one the key features of Algorithm~\ref{alg:fullRR} with respect to the genuine Robin-Robin explicit coupling scheme reported in \cite[Algorithm 4]{burman2014explicit}, in which 
$ \alpha = \gamma \mu/h$. The resulting splitting error  scales as $\mathcal O(\Delta t/h)$, and 
hence preventing convergence under 
$\Delta t = \mathcal O(h)$. Figure \ref{fig:comp-genuine} provides numerical evidence of this issue and shows  
Algorithm~\ref{alg:fullRR}  fixes it.

The $h$-uniformity of the splitting error has further implications in terms of  accuracy.  Indeed, owing 
to Theorem~\ref{mainthm2}, one correction iteration in Algorithm~\ref{alg:fullRR} should be
 enough to achieve overall $\mathcal O(h)$ accuracy, under $\Delta t = \mathcal O(h)$. 
Numerical evidence of this  is given in Figure \ref{fig:err-time-1c}. We can also notice 
that the convergence behavior is very close to the one provided by the strongly coupled scheme. 
This is a fundamental advantage of Algorithm~\ref{alg:fullRR} with respect to the genuine Robin-Robin explicit coupling scheme,
in which both high order extrapolation and several corrections are need to coped with the loss of $h$-uniformity (see \cite{burman2014explicit}).

The superior accuracy of the Algorithm~\ref{alg:fullRR} with one correction iteration  is also clearly visible
in Figure \ref{fig:disp}, where  the interface displacements associated to Figures \ref{fig:err-time} and 
 \ref{fig:err-time-1c} (first four points of each curve) are displayed. 
For comparison purposes, the reference displacement is also shown. 
Observe that the defect-correction variant of Algorithm~\ref{alg:fullRR} retrieves the accuracy of the strongly coupled  scheme.

\subsection{Impact of the Robin coefficient $\alpha$}
We now turn our attention to another fundamental question related to Algorithm~\ref{alg:fullRR}: 
the choice of the Robin parameter $\alpha$. From Theorem~\ref{mainthm2},  
the leading term of the time splitting error scales as 
$$\sqrt{\alpha^{-1} + \alpha}\sqrt{\Delta t}.$$ 
We can hence anticipate that  accuracy should  be spoiled for (relatively) large or small values of $\alpha$. 
Numerical evidence of this behavior is provided in Figures \ref{fig:err-time-alpha} and  \ref{fig:err-time-1c-alpha}, 
where the convergence histories obtained with Algorithm~\ref{alg:fullRR} (without and with correction) are reported 
for different values of $\alpha$. Indeed, the best accuracy is obtained for (relatively) 
moderate values of  $\alpha$, ranging from  250 to 2000, whereas out of this interval accuracy degrades rapidly. 
It should be noted that, since $\alpha$ is not dimensionless, these optimal values are expected to depend on the physical 
parameters of the system. 

\begin{figure}[h!]
\centering
\includegraphics[width=0.6\linewidth]{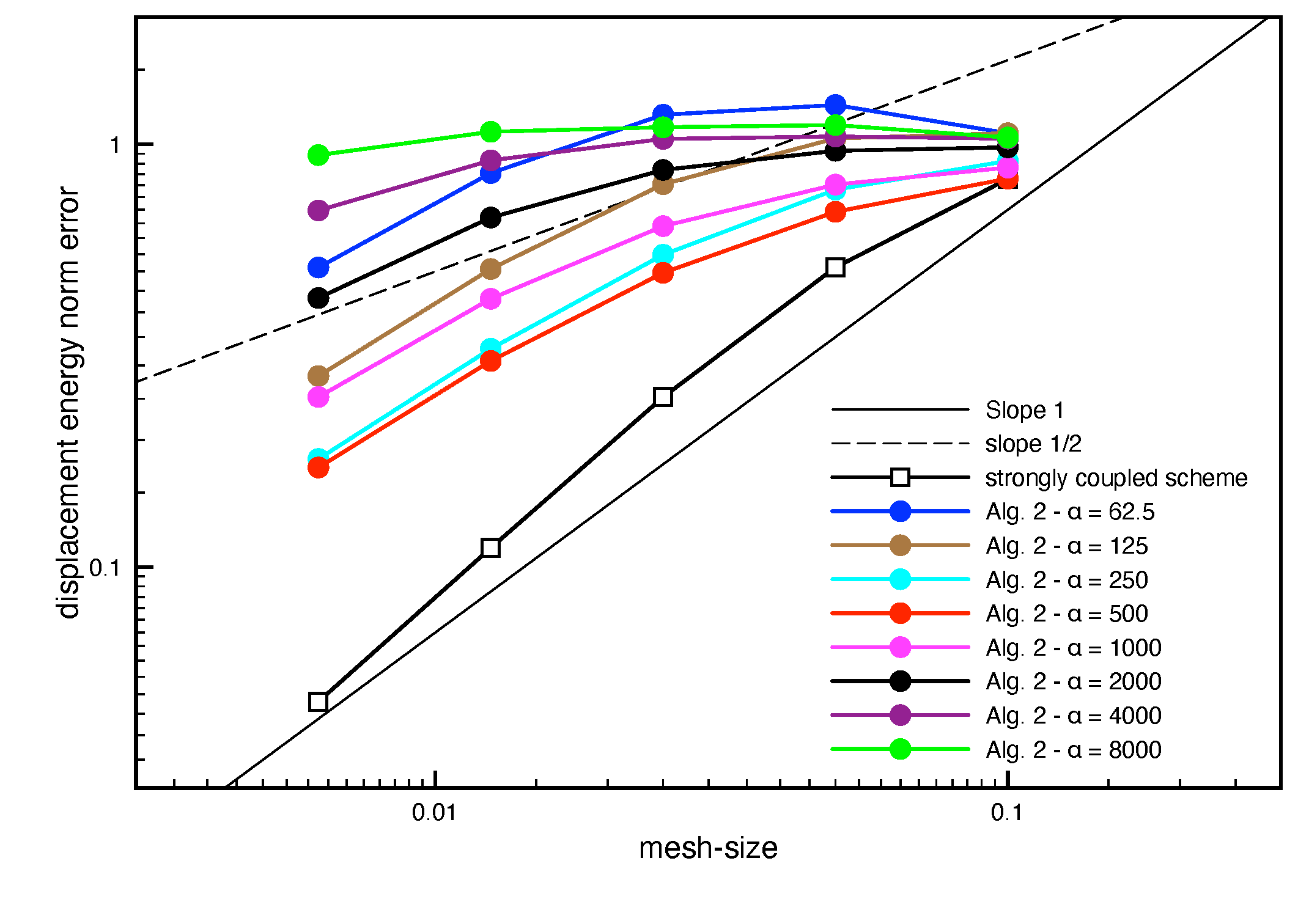} 
\caption{Time-convergence history of the displacement at $t=0.015$, with $\Delta t = \mathcal O (h)$ obtained with Algorithm~\ref{alg:fullRR}  for 
different values of $\alpha$.}
\label{fig:err-time-alpha}
\end{figure}
 \begin{figure}[h!]
	\centering
\includegraphics[width=0.6\linewidth]{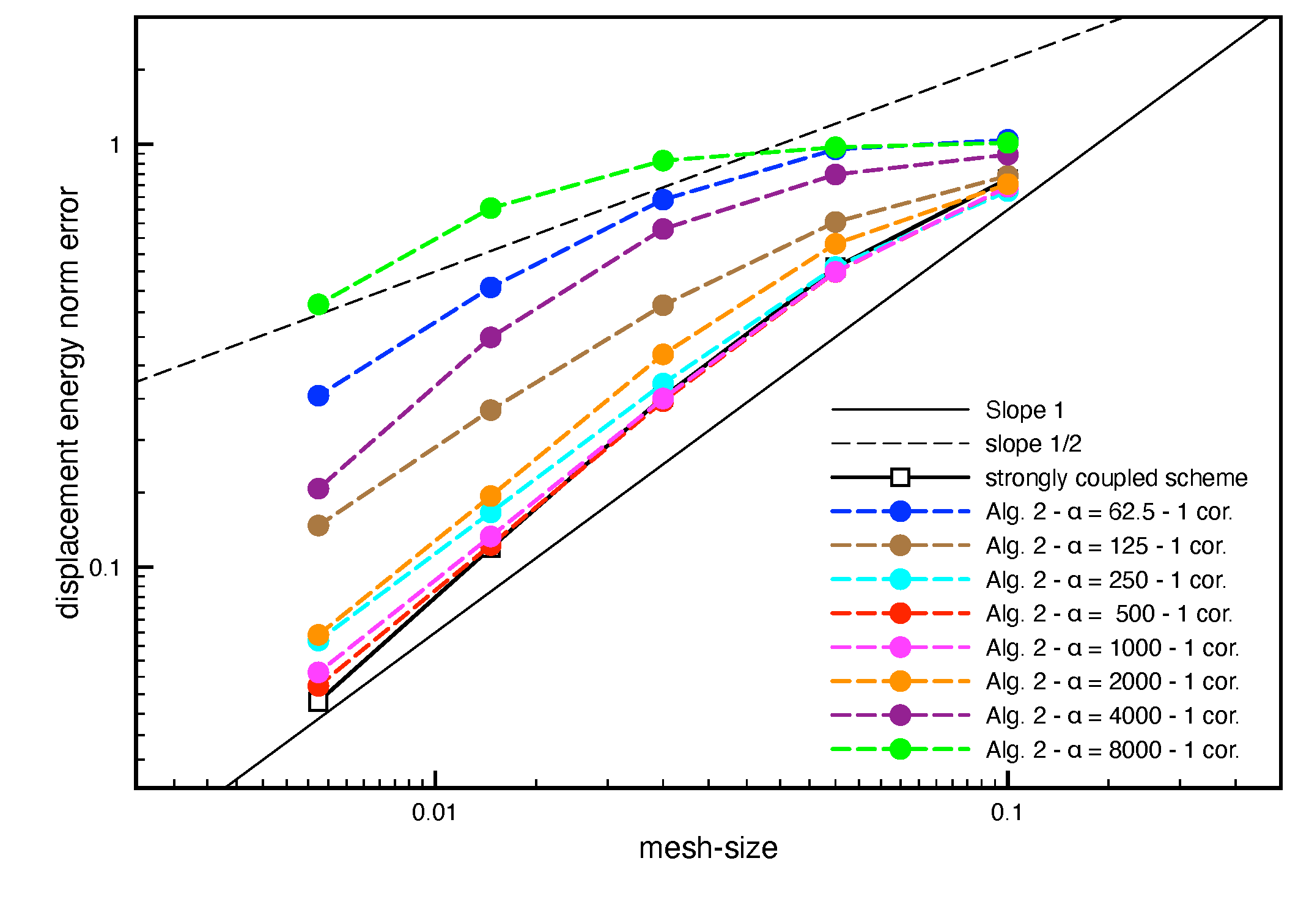} 
\caption{Time-convergence history of the displacement at $t=0.015$, with $\Delta t = \mathcal O (h)$ obtained with Algorithm~\ref{alg:fullRR} 
with 1 correction iteration for different values of $\alpha$.}
\label{fig:err-time-1c-alpha}
\end{figure}


\bibliographystyle{abbrv}
\bibliography{references}

\appendix
\section{Proof of Lemma \ref{ineqLemma}}

\begin{proof}

We begin this proof by establishing several identities. The first identity follows from Minkowski's integral inequality and Jensen's inequality. Namely, we note that for any $H^\ell$-norm $\|\cdot \|$, we have
\begin{equation}\label{normBound}
\|\frac{1}{\Dt}\int_{t_n}^{t_{n+1}} w(\cdot, s)ds \|^2 \leq \frac{1}{\Dt} \int_{t_n}^{t_{n+1}} \| w(\cdot,s)\|^{2} ds.
\end{equation}
Our remaining identities are straightforward to prove. Let $|r| \leq 2$ and define 
\begin{equation*}
\bar{w}^{n+1}(x):= \frac{1}{\Dt} \int_{t_{n}}^{t_{n+1}} w(x, s) \, ds.
\end{equation*}
We have
\begin{alignat}{1}
\partial_{x}^r (w^{n+1} - w^{n}) &= \int_{t_n}^{t_{n+1}} \partial_{x}^{r} (\pt w)(\cdot, s)ds, \label{id1} \\
\partial_{x}^r( \pdt w^{n+1} - \partial_{t}w^{n+1/2} ) & = \frac{1}{2\Dt} \int_{t_{n}}^{t_{n+1}}(t_{n+1}-s) (s-t_{n})\partial_{x}^r \partial_{t}^{3}w(\cdot,s)ds,\label{id2} \\
 \pdt w^{n+1}-\pt w^{n+1} &= \frac{-1}{\Dt} \int_{t_n}^{t_{n+1}} (s-t_n) \partial_t^2 w(\cdot, s) ds, \label{id3} \\
\partial_{x}^r(w^{n+1} - \bar{w}^{n+1}) &= \frac{1}{\Dt}\int_{t_n}^{t_{n+1}}(s-t_n)\partial_{x}^{r}( \pt w)(\cdot, s)ds, \label{id4} \\
\partial_{x}^r(w^{n+1/2} - \bar{w}^{n+1}) &= \frac{-1}{2\Dt}\int_{t_n}^{t_{n+1}}(t_{n+1} - 2s + t_n)\partial_{x}^{r}( \pt w)(\cdot, s)ds.\label{id5}
\end{alignat}

We may now proceed with the proof of Lemma \ref{ineqLemma}.


To prove \eqref{tt1}, we write $\pdt R_h^s \bq^{n+1}-\pt \bq^{n+1/2}= (R_h^s -I) \pdt \bq^{n+1}+( \pdt \bq^{n+1} - \partial_{t}\bq^{n+1/2} )$ and
\begin{equation*}
 \pdt \bq^{n+1}= \frac{1}{\Dt} \int_{t_{n}}^{t_{n+1}}  (\partial_t \bq)(\cdot, s) ds.
\end{equation*}
Hence, 
\begin{equation*}
(R_h^s - I) \pdt \bq^{n+1}=\frac{1}{\Dt} \int_{t_{n}}^{t_{n+1}}  (R_h^s-I)(\partial_t \bq)(\cdot, s) ds.
\end{equation*}
It therefore follows from \eqref{normBound} and \eqref{eq1} that
\begin{equation}\label{r51}
\|(R_h^s - I)  \pdt \bq^{n+1}\|_{L^2(\Omega_s)}^2 \le \frac{1}{\Dt} \int_{t_{n}}^{t_{n+1}}  \|(R_h^s - I)(\partial_t \bq)(\cdot, s)\|_{L^2(\Omega_s)}^2 ds \le \frac{C h^{4}}{\Dt} \int_{t_{n}}^{t_{n+1}}  \|\partial_t \bq(\cdot, s)\|_{H^2(\Omega_s)}^2 ds .
\end{equation}
To estimate $\pdt \bq^{n+1} - \partial_{t}\bq^{n+1/2}$, we apply H\"older's inequality to \eqref{id2} with $|r|=0$ to obtain
\begin{alignat*}{1}
|\pdt \bq^{n+1} - \partial_{t}\bq^{n+1/2}|\leq & \left(\frac{1}{4\Dt^{2}} \int_{t_{n}}^{t_{n+1}} (t_{n+1}-s)^{2}(s-t_{n})^{2}ds \right)^{1/2}\left( \int_{t_{n}}^{t_{n+1}}|\partial_{t}^{3}\bq(\cdot,s)|^2 ds\right)^{1/2} \\
=&\left( \frac{\Dt^3}{5!}\right)^{1/2}\left( \int_{t_{n}}^{t_{n+1}}|\partial_{t}^{3}\bq(\cdot,s)|^2 ds\right)^{1/2}.
\end{alignat*}
Therefore,
\begin{equation}\label{j64}
\|\pdt \bq^{n+1} - \partial_{t}\bq^{n+1/2}\|_{L^2(\Omega_s)}^2 \leq \frac{\Dt^3}{5!} \int_{t_{n}}^{t_{n+1}}\|\partial_{t}^{3}\bq(\cdot,s)\|_{L^2(\Omega_s)}^2ds.
\end{equation}


To get the estimate \eqref{tt2}, we recall that $\blam^{n+1} =\sigma_{f}(\bu^{n+1},p^{n+1})\bn$. Then, we use a trace inequality \eqref{trace} to get 
\begin{alignat*}{1}
 \|\blam^{n+1}-\blam^n\|_{L^2(\Sigma)}^2 &\le C\|\sigma_{f}(\bu^{n+1},p^{n+1})\bn-\sigma_{f}(\bu^{n},p^{n})\bn\|_{H^1(\Omega_f)}^2 \\
&\le C\bigg(\mu^2 \|\epsilon(\bu^{n+1} - \bu^{n})\|_{H^1(\Omega_f)}^2 + \|p^{n+1}-p^{n}\|_{H^1(\Omega_f)}^2 \bigg).
\end{alignat*}
Applying \eqref{normBound} and \eqref{id1}, we obtain our result.


The bound for \eqref{tt4} follows in the same manner with an additional application of \eqref{Stability}.

Similarly, \eqref{tt3} follows immediately from \eqref{id2} and \eqref{normBound} after applying the trace inequality \eqref{trace} and the stability result \eqref{Stability}.
\\


To get the bound \eqref{tt5} we write $\pdt R_h^f \bu^{n+1}- \pt \bu^{n+1}= (R_h^f-I) \pdt \bu^{n+1}+ \pdt \bu^{n+1}-\pt \bu^{n+1}$. Similar to the bound \eqref{r51}, we can show that 
\begin{equation*}
\|(R_h^f-I) \pdt \bu^{n+1}\|_{L^2(\Omega_f)}^2 \le  \frac{C h^{4}}{\Dt} \Big(\int_{t_{n}}^{t_{n+1}}  \|\partial_t \bu(\cdot, s)\|_{H^2(\Omega_f)}^2 \Big). 
\end{equation*}
Furthermore, applying \eqref{id4} and H\"{o}lder's inequality, we establish
 \begin{equation*}
\|\pdt \bu^{n+1}-\pt \bu^{n+1}\|_{L^2(\Omega_f)}^2 \le  \frac{\Dt}{3}  \int_{t_n}^{t_{n+1}} \|\partial_t^2 \bu(\cdot, s)\|_{L^2(\Omega_f)}^2 ds. 
 \end{equation*}
 Combining the above two inequalites gives  \eqref{tt5}. 
\\


 For \eqref{tt6}, we use \eqref{Stability} and write 
\begin{equation*}
\|\nabla R_h^s \bg_2^{n+1}\|_{L^2(\Omega_s)}^2 \leq C \|\bg_2^{n+1}\|_{H^1(\Omega_s)}^2.
\end{equation*}
The bound now follows exactly that of \eqref{tt3}.

In order to prove \eqref{tt7} we bound each term in $\bg_1^{n+1}$ separately. To bound the term $\alpha(\pdt R_h^s \bfeta^{n+1}- R_h^f \bu^{n+1})$ we use that $\bu^{n+1}= \pt \bfeta^{n+1}$ on $\Sigma$, indicating that $R_h^f \bu^{n+1} = R_h^s \pt \bfeta^{n+1}$. Therefore
 \begin{equation*}
\alpha(\pdt R_h^s \bfeta^{n+1}- R_h^f \bu^{n+1})=\alpha( \pdt R_h^s \bfeta^{n+1} - \pt R_h^s \bfeta^{n+1}).
\end{equation*}
Thus, applying the trace inequality \eqref{trace} and stability \eqref{Stability}, we have
\begin{equation*}
\alpha^{2}\|\pdt R_h^s \bfeta^{n+1} - \pt R_h^s \bfeta^{n+1} \|_{L^2(\Sigma)}^2 \leq C \alpha^2 \|\pdt \bfeta^{n+1} - \pt \bfeta^{n+1} \|_{L^2(t_n, t_{n+1}; H^1(\Omega_s))}^2.
\end{equation*}
Therefore, applying \eqref{id3} and using H\"{o}lder's inequality, we have
\begin{equation*}
\alpha^{2}\|\pdt R_h^s \bfeta^{n+1} - \pt R_h^s \bfeta^{n+1} \|_{L^2(\Sigma)}^2 \leq C \alpha^2 \Dt  \int_{t_n}^{t_{n+1}} \|\partial_t^2 \bfeta(\cdot, s)\|_{H^1(\Omega_s)}.
\end{equation*}
Combining this with \eqref{tt3} gives \eqref{tt7}.

Next, the proofs for \eqref{tt12} and \eqref{tt13}are nearly identical, so we only provide the proof of \eqref{tt12}.  Recall from \eqref{eq1} that
\begin{equation*}
\|S_h p^{n+1} - p^{n+1}\|_{L^{2}(\Omega_f)}^2 \le Ch^4 \|p^{n+1}\|_{H^{2}(\Omega_f)}^2.
\end{equation*}

We may then write $p^{n+1} = p^{n+1} - \bar{p}^{n+1} + \bar{p}^{n+1}$. Thus we have
\begin{equation*}
\|S_h p^{n+1} - p^{n+1}\|_{L^{2}(\Omega_f)}^2 \le Ch^4 \bigg(\|p^{n+1} - \bar{p}^{n+1}\|_{H^{2}(\Omega_f)}^2 + \|\bar{p}^{n+1}\|_{H^{2}(\Omega_f)}^2\bigg).
\end{equation*}
 Then, using \eqref{id4} \eqref{normBound}, along with H\"{o}lder's inequality, we have
\begin{alignat*}{1}
\|p^{n+1} -\bar{p}^{n+1}\|_{H^{2}(\Omega_f)}^2 \le& C\Dt \|\pt p\|_{L^{2}(t_n, t_{n+1} ;H^{2}(\Omega_f))}^2,\\
\|\bar{p}^{n+1}\|_{H^{2}(\Omega_f)}^2 \le& C\frac{1}{\Dt}\| p\|_{L^{2}(t_n, t_{n+1} ;H^{2}(\Omega_f))}^2.
\end{alignat*}
Our result then follows from combining the terms above.



For \eqref{tt11}, we use the stability result \eqref{Stability} to recognize that $\|\nabla S_h p^{n+1} \|_{L^2(\Omega_f)} \leq C \|p^{n+1}\|_{H^1(\Omega_f)}$. We may then follow the proof of \eqref{tt12} to write
\begin{equation*}
\|\nabla S_h p^{n+1} \|_{L^2(\Omega_f)}^2 \leq C\bigg( \|p^{n+1} - \bar{p}^{n+1}\|_{H^1(\Omega_f)}^2 + \| \bar{p}^{n+1}\|_{H^1(\Omega_f)}^2 \bigg).
\end{equation*}

Following the same process as \eqref{tt12}, this yeilds 
\begin{equation*}
\|\nabla S_h p^{n+1} \|_{L^2(\Omega_f)}^2 \leq C\bigg( \Dt \|\pt p\|_{L^{2}(t_n, t_{n+1} ;H^{1}(\Omega_f))}^2 +\frac{1}{\Dt}\| p\|_{L^{2}(t_n, t_{n+1} ;H^{1}(\Omega_f))}^2 \bigg).
\end{equation*}
In a similar fashion, we bound \eqref{tt16} by writing $\bu^{n+1} = \bu^{n+1} - \bar{\bu}^{n+1} + \bar{\bu}^{n+1}$. The result follows in the same manner as \eqref{tt11}.


To prove \eqref{tt9}, we follow the proof of \eqref{tt12}, however we apply \eqref{eq2} and \eqref{id5} in place of \eqref{eq1} and \eqref{id4}. Thus we have
\begin{equation*}
\|\nabla (R_h^s - I) \bfeta^{n+1/2}\|_{L^2(\Omega_s)}^2 \leq Ch^2 \|\bfeta^{n+1/2}\|_{H^2(\Omega_s)}^2.
\end{equation*}

Thus, noting $\bfeta^{n+1/2} = \bfeta^{n+1/2} - \bar{\bfeta}^{n+1} + \bar{\bfeta}^{n+1}$, we have

\begin{alignat*}{1}
 \|\bfeta^{n+1/2} -\bar{\bfeta}^{n+1}\|_{H^2(\Omega_s)}^2 \le& C\Dt \|\pt \bfeta\|_{L^{2}(t_n, t_{n+1} ;H^{2}(\Omega_s))}^2,\\
\|\bar{\bfeta}^{n+1}\|_{H^{2}(\Omega_f)}^2 \le& \frac{1}{\Dt}\|\bfeta\|_{L^{2}(t_n, t_{n+1} ;H^{2}(\Omega_s))}^2.
\end{alignat*}
We then combine terms.


Finally, for \eqref{tt15} we use \eqref{normBound} and \eqref{eq2}, to get
\begin{alignat*}{1}
\|\nabla(R_h^s -I)\pdt \bfeta^{n+1/2}\|_{L^2(\Omega_s)}^2 \leq& C \frac{h^2}{\Dt}\|\pt \bfeta\|_{L^2(t_{n},t_{n+1}; H^2(\Omega_s))}^2.
\end{alignat*}

\end{proof}

\end{document}